\documentclass[11pt,a4paper,reqno]{amsart}
\usepackage[english]{babel}
\usepackage[utf8]{inputenc} 
\usepackage{fancyhdr}
\usepackage{indentfirst}
\usepackage{color}     
\usepackage{graphicx}   
\usepackage{newlfont}

\usepackage{amssymb}
\usepackage{amsmath}
\usepackage{latexsym}
\usepackage{amsthm}
\usepackage{mathrsfs}
\usepackage{verbatim}
\usepackage[all]{xy}
\usepackage{lineno}
\usepackage{hyperref}

\theoremstyle{plain} 
\newtheorem*{theo*}{Theorem}
\newtheorem{theo}{Theorem}[section] 
\newtheorem{prop}[theo]{Proposition}
\newtheorem{cor}[theo]{Corollary}
\newtheorem{lem}[theo]{Lemma}
\newtheorem*{con*}{Conjecture}

\theoremstyle{definition}
\newtheorem{defin}[theo]{Definition}
\newtheorem{ex}[theo]{Example}

\theoremstyle{remark}
\newtheorem{rem}[theo]{Remark}

\newcommand{\bilinear}{(\cdot|\cdot)}
\newcommand{\barspace}{\hspace{0.1cm}\left| \hspace{0.1cm}}

\newcommand{\Z}{\mathbb{Z}}
\newcommand{\Zplus}{\Z_{>0}}
\newcommand{\Zpluseq}{\Z_{\geq0}}
\newcommand{\R}{\mathbb{R}}
\newcommand{\T}{\mathbb{T}}
\newcommand{\C}{\mathbb{C}}
\newcommand{\A}{\mathcal{A}}

\newcommand{\Uone}{\operatorname{U}(1)}
\newcommand{\SU}{\operatorname{SU}(2)}
\newcommand{\SO}{\operatorname{SO}(3)}
\newcommand{\Diff}{\operatorname{Diff}}
\newcommand{\Vir}{\mathfrak{Vir}}

\begin{document}

\author{Sebastiano Carpi}
\thanks{S.C. is supported in part by the ERC advanced grant 669240 QUEST ``Quantum Algebraic
Structures and Models'' and GNAMPA-INDAM}
\address{Dipartimento di Economia, Universit\`a di Chieti-Pescara ``G. d'Annunzio'', Viale Pindaro, 42, 65127 Pescara, Italy\\
E-mail: {\tt s.carpi@unich.it}}
\author{Tiziano Gaudio}
\address{Department of Mathematics and Statistics, Lancaster University, Lancaster LA1 4YF, UK\\
E-mail: {\tt t.gaudio@lancaster.ac.uk}}
\author{Robin Hillier} 
\address{Department of Mathematics and Statistics, Lancaster University, Lancaster LA1 4YF, UK\\
E-mail: {\tt r.hillier@lancaster.ac.uk}}
\title[Classification of Unitary Vertex Subalgebras and Conformal Subnets]{Classification of Unitary Vertex Subalgebras and Conformal Subnets for Rank-One Lattice Chiral CFT Models}
\date{August 26, 2019}

\begin{abstract}
We provide a complete classification of unitary subalgebras of even rank-one lattice vertex operator algebras. As a consequence of the correspondence between vertex operator algebras and conformal nets, we also obtain a complete classification of conformal subnets of even rank-one lattice conformal nets. 
\end{abstract}

\maketitle

\begin{section}{Introduction}

Two-dimensional conformal quantum field theory (2D CFT) is an exciting topic, both from the physical and from the mathematical point of view, with a variety of applications, cf. \cite{DiFMS1996,EK1998,Gannon2006} for some of them. The special subclass of chiral CFTs, namely those 2D CFTs containing right-moving fields only or left-moving fields only, are the building blocks of 2D CFT. These theories can be considered as quantum field theories on the real line (the light-ray) or, after compactification, on the circle $S^1$. 

There are mainly two mathematical axiomatizations of chiral CFT, one of them goes via conformal nets, the other one via vertex operator algebras (VOAs). 

A conformal net is a family of von Neumann algebras indexed by the intervals of $S^1$. These algebras should be regarded as the algebras generated by bounded versions of smeared fields localized in the respective interval. Moreover, the family should satisfy a number of axioms motivated by physics, in particular it should carry an action of the (orientation-preserving) diffeomorphism group $\Diff^+(S^1)$. The overall framework in this approach is that of algebraic quantum field theory (AQFT), cf. \cite{Haag} for an overview.

The second approach goes via vertex operator algebras. A vertex operator is a certain formal Laurent series and is thought of as a quantum field. Together they give rise to a vertex operator algebra, which formally consists of a vector space together with a state-field correspondence and some additional structure necessary to make things work. This additional structure is an axiomatization of physical properties but the axioms here are different from the ones of conformal nets and can be seen rather as an algebraic formulation of the Wightman axioms, see e.g. \cite{Kac1997}.

Various problems that can be solved in the conformal net approach are still open in the VOA approach and vice versa. It is generally believed that these two framework should be more or less equivalent in the unitary case but for a long time a direct general connection between those two approaches had not been clear. However, a significant step forward has been achieved recently in \cite{CKLW2015} where a map from a suitable class of unitary VOAs, the class of strongly local VOAs,  to the class of conformal nets has been defined. Moreover, many examples of unitary VOAs have been shown to be strongly local and it is conjectured that the map gives in fact a one-to-one correspondence between (simple) unitary VOAs and (irreducible) conformal nets, see also \cite{Kaw2015} for an overview of these results and related conjectures. Further progress in the connection of these two mathematical frameworks for chiral CFT have been made in \cite{CWX,Gui2017a,Gui2017b,Tener2016,Tener2018}.

In both approaches there is a natural notion of subtheory of a given chiral CFT. These subtheories are called conformal subnets or unitary vertex subalgebras (or simply unitary subalgebras) depending on the chosen framework. For a conformal net $\mathcal{A}_V$ that comes from a strongly local VOA $V$, it has been shown in \cite{CKLW2015} that the conformal subnets $\mathcal{A}\subset\mathcal{A}_V$ are in one-to-one correspondence with the unitary subalgebras $W\subset V$. 
The study and the classification of these subtheories is a very natural problem in either of the two approaches. One of the reasons is that subtheories give rise to new models from known ones. Typical examples are the orbifold models or the coset models but the study of subtheories could give new exotic examples besides these standard constructions.  For example Evans and Gannon in \cite{EG2011} argued in favour of the existence of ``exotic'' subtheories of certain affine Lie algebra chiral CFT models that should be related in an appropriate sense to the Haagerup subfactor. 

In the AQFT setting the structure of (Haag dual) subnets of a given local  Poincar\'e covariant net on the four-dimensional Minkowski space-time is quite well understood under rather mild and natural assumptions \cite{CC2001a,CC2001b,CC2005,CDR2001}. Here the subnet structure can be completely described in terms of fixed point subnets under actions of compact groups of inner symmetries (compact orbifolds) and tensor product decompositions, after replacing the net, if necessary, with its maximal extension namely the canonical field net of Doplicher and Roberts
\cite{DoRo90}. Moreover, starting from the general description one can describe the subnet structure in a very explicit way for many free field models.  These results rely heavily on Doplicher-Roberts abstract duality for compact groups \cite{DoRo89a,DoRo89b}, on which the construction of the canonical field net in \cite{DoRo90} is based. In particular, they rely on the permutation symmetry of the superselection sectors and hence they do not apply in low space-time dimensions because of the presence of nontrivial braid group relations. 

In the conformal net setting there are various abstract results on the general structure of conformal subnets, see e.g. \cite{Bischoff2017b,CKL2010,DVG2017,JonesXu04,Longo2003,LR1995,Xu2014} but it is not clear how to use them in order to get explicit classification results in concrete models. For example it has been shown in \cite{Bischoff2017b} that the finite index subnets of a given net can be completely described in terms of hypergroup actions but it is not clear how to determine the structure of these hypergroup actions without knowing the subnet structure a priori. 

Actually, only few results appear to be known for the classification of conformal subnets in concrete models, cf. \cite{Carpi1998,Carpi1999} and \cite[Subsection 6.3]{KL2004}.  In more detail, it is known that a Virasoro net has no proper nontrivial conformal subnet \cite{Carpi1998},  that all the nontrivial subnets of the loop group net associated to the level one vacuum representation of the loop group for $\mathrm{SU}(2)$ arise as compact orbifolds  for closed subgroups of the automorphism group $\mathrm{SO}(3)$ of the net \cite{Carpi1999}. Moreover, the finite index conformal subnets of all the conformal nets with central charge $c<1$ are classified in \cite{KL2004}. The results in \cite{Carpi1998,Carpi1999} have a (unitary) VOA analogue \cite{CKLW2015}, see also \cite{DongGriess1998}. This is not surprising in view of the fact than the underlying VOAs are known to be strongly local so one can directly apply the correspondence between conformal subnets and unitary subalgebras. Concerning the unitary VOAs with $c<1$,  strong locality results are presently available only in some special cases but it should be possible to obtain the analogous results of \cite[Subsection 6.3]{KL2004} directly from the classification of pre-unitary VOAs with $c<1$ given in \cite{DongLin2015}.

One immediate example of a $c=1$ model consists of the $\Uone$-current algebra CFT, also called free boson CFT. Its extensions have been classified in \cite{BMT1988} in the conformal net setting. They can be expressed as the conformal nets $\A_{\Uone_{2N}}$ or as the simple unitary VOAs $V_{L_{2N}}$, respectively, associated to an even rank-one lattice $\alpha\Z$ with generator $\alpha$ satisfying $(\alpha|\alpha)=2N$, where $N\in\Zplus$ labels the lattices uniquely. 
These models have been studied e.g. in \cite{BMT1988,Xu2005,DX2006} or \cite{Dong1993,DongGriess1998}, respectively.

It is widely believed that these models, together with their (compact, possibly finite) orbifolds exhaust all possible unitary chiral CFTs with $c=1$, see e.g. \cite{Dijk1989,Ginsparg,Kiritsis}.  This fact has been proved in the conformal net setting by Xu \cite {Xu2005} under an assumption called ``Spectrum Condition'', see also \cite{Carpi2004} for some special cases of these results. This assumption is verified in all the above $c=1$ conformal nets but its general validity is still an open problem. 
On the VOA side recent progress on the classification of rational unitary $c=1$ chiral CFTs has been made in 
\cite{DongJiang2013,DongJiang2015} but again the problem is still open. 

In this paper we completely classify the conformal subnets and the unitary vertex subalgebras of all currently known chiral CFT models with central charge $c=1$. We first give the classification in the unitary VOA framework. We only use (unitary) VOA arguments without reference to strong locality or conformal nets.  To this end we rely on previous results by Dong and Griess \cite{DongGriess1998}. 
We show that if $N$ is not the square of a positive integer and $N\neq 2$ then every nontrivial unitary vertex subalgebra of $V_{L_{2N}}$  is either an orbifold  $V_{L_{2N}}^H$ for some closed subgroup of the compact group of unitary VOA automorphisms $D_\infty$  of $V_{L_{2N}}$ or it coincides with the Virasoro subalgebra generated by the $c=1$ conformal vector. If $N=2$, besides the compact orbifolds and the $c=1$ Virasoro subalgebra there is a further explicitly described continuous family  
$\{W_t:\, t\in \T\}$ of unitary subalgebras, all isomorphic to the $c=1/2$ simple Virasoro VOA, but there are no further nontrivial unitary subalgebras. For $N=1$ it is well known that $V_{L_{2N}}$ is isomorphic to the affine unitary vertex operator algebra $V_{\mathfrak{sl}(2,\mathbb{C})_1}$ associated to $\mathfrak{sl}(2,\mathbb{C})$ at level one. We show that 
every nontrivial unitary subalgebra is an orbifold  $V_{L_{2}}^H$ for some closed subgroup of the compact group of unitary VOA automorphisms $\mathrm{SO}(3)$  of $V_{L_{2}}$. As already mentioned the latter result was already known but our proof is slightly different from the one given in \cite{CKLW2015} in order to have a purely VOA argument. Finally, for every integer $k>1$, the classification of the unitary subalgebras $V_{L_{2k^2}}$ follows from the unitary embedding 
$V_{L_{2k^2}} \subset V_{L_{2}}$.  The results on the classification of conformal subnets are then obtained using the correspondence between simple unitary subalgebras and conformal subnets given in \cite{CKLW2015} and the fact that the rank-one lattice VOAs are known to be strongly local. In particular, our results show that no new $c=1$ chiral CFT model can be found by taking subtheories of the currently known $c=1$ chiral CFTs, in agreement with the common belief.

This paper is organized as follows. We start in Section \ref{section:prelim} by recalling some basics about unitary VOAs in general and the rank-one lattice type models in particular. Section \ref{section:compl_class} deals with the main part of the classification in the VOA setting. Finally, Section \ref{section:conformal_nets} translates those results into the language of conformal nets, recalling and using the correspondence between the two settings developed in \cite{CKLW2015}. 

\end{section}

\begin{section}{Preliminaries}\label{section:prelim}
This section is dedicated to setting the notation we use throughout the paper and to introducing some preliminaries about unitary vertex operator algebras (VOAs) and even rank-one lattice type VOAs. References are given at the beginning of every subsection when required.


\begin{subsection}{Unitary vertex operator algebras}  \label{subsec:unitary_voas}
The present introduction to unitary vertex operator algebras follows mostly the structure and the notation given in \cite[Section 4, 5]{CKLW2015}, see also \cite[Chapter 2, 6]{LepLi2012}, \cite[Chapter 4]{Kac1997}.

Given a vector space $V$ on $\mathbb{C}$, consider the space of \textit{doubly-infinite formal Laurent series} in the variable $z$ on $\mathbb{C}$ with coefficients in $\mathrm{End}V$ (see \cite[Chapter 2]{LepLi2012})
\begin{equation*}
\left(\mathrm{End}V\right)\left[[z,z^{-1}]\right]:=
\left\{ a(z):=\sum_{n\in\mathbb{Z}}a_{(n)}z^{-n-1}
\mid
a_{(n)}\in\mathrm{End}V
 \right\} .
\end{equation*}

Then, a \textit{vertex algebra} is defined by a quadruple $(V,\Omega, Y, T)$:

\begin{itemize}
\item $\Omega\in V$ is called the \textit{vacuum vector} of $V$
\item $T\in\mathrm{End}V$ is called the \textit{infinitesimal translation operator} of $V$. 
\item $Y$ is a linear map between $V$ and $\left(\mathrm{End}V\right)\left[[z,z^{-1}]\right]$ called the \textit{state-field correspondence}, defined by
\begin{equation*}
a\longmapsto Y(a,z):=\sum_{n\in\mathbb{Z}}a_{(n)}z^{-n-1}
\end{equation*}
with the following properties:
\begin{itemize}
\item \textit{(Field)}. For any $a \in V$, $Y(a,z)$ is a \textit{field}, namely, for all $b \in V$ there exists an integer $N\geq0$ such that $a_{(n)}b=0$ for all $n\geq N$.
\item \textit{(Translation covariance)}. $[T,Y(a,z)]=\frac{\mathrm{d}}{dz}Y(a,z)$ for all $a\in V$.
\item \textit{(Vacuum)}. $T\Omega=0$, $Y(\Omega, z)=I_V$, $a_{(-1)}\Omega= a$ for all $a\in V$.
\item \textit{(Locality)}. For all $a, b\in V$, there exist $N>0$ such that (see \cite[Chapter 2]{Kac1997} for formal calculus for fields)
\begin{equation*}
(z-w)^N[Y(a,z), Y(b,w)]=0 .
\end{equation*}
\end{itemize}
We call such a field $Y(a,z)$ a \textit{vertex operator}.
\end{itemize}

A \textit{vertex operator algebra} or just \textit{VOA} is a vertex algebra $\left(V, \Omega, Y, T\right)$ with an element $\nu\in V$, called \textit{the conformal vector} of $\left(V, \Omega, Y, T\right)$, such that:

\begin{itemize}
\item $\nu$ is a \textit{Virasoro vector with central charge} $c\in\C$, which means that $Y(\nu, z)=\sum_{n\in\mathbb{Z}}L_nz^{-n-2}$, where $L_n:=\nu_{(n+1)}$, is a \textit{Virasoro field}, i.e.,  the endomorphisms $L_n$ satisfy the following commutation relations:
\begin{equation}  \label{eq:VirasoroCRs}
\left[L_m, L_n\right]=(m-n)L_{m+n}+ \frac{c(m^3-m)}{12}\delta_{m,-n}
\quad\forall m, n\in\mathbb{Z} 
 \, .
\end{equation}
\item $\nu$ is a \textit{conformal vector}, which means that $L_{-1}=T$ and $L_0$ is diagonalizable on $V$. In this case, $Y(\nu,z)$ and $L_0$ are called the \textit{energy-momentum field} and the \textit{conformal Hamiltonian} of $(V, \Omega, Y, T, \nu)$ respectively.
\item The following conditions must hold:
\begin{itemize}
\item $V=\bigoplus_{n\in\mathbb{Z}}V_n$, where $V_n:=\mathrm{Ker}\left\{L_0-n I_V\right\}$;
\item $\dim V_n<\infty$;
\item there exists an $N\in\mathbb{Z}$ such that $V_n=0$ for all $n\leq N$.
\end{itemize}
\end{itemize}

We will often indicate a VOA $(V, \Omega, Y, T, \nu)$ just with the vector space $V$.

An \textit{antilinear/linear automorphism} $g$ of a VOA $V$ is a vector space antilinear/linear automorphism such that 
\begin{equation*}
g(\Omega)=\Omega, \hspace{0.3cm}
g(\nu)=\nu, \hspace{0.3cm}
g(a_{(n)}b)=g(a)_{(n)}g(b)   \hspace{0.3cm}
\forall a, b\in V.
\end{equation*}
We define $\mathrm{Aut} V$ the group of linear automorphisms of $V$, which has a natural structure of a topological group (see \cite[Section 4.3]{CKLW2015}). In a similar manner, one defines a \textit{linear isomorphism} between two VOAs.

A VOA $V$ is of \textit{CFT type} if $V_n=0$ for all $n<0$ and $V_0=\mathbb{C}\Omega$. 

If $a\in V_n$, we say that $a$ is \textit{homogeneous} of \textit{conformal weight} $d_a:=n$. In this case, we also use the following notation:
\begin{equation*}
Y(a,z)=\sum_{n\in\mathbb{Z}} a_nz^{-n-d_a} ,
\hspace{0.5cm}
a_n:=a_{(n+d_a-1)}.
\end{equation*}

We call an element $a\in V$ \textit{primary} if $L_na=0$ for all $n\geq 1$ and \textit{quasi-primary} if $L_1a=0$. Therefore, using the fact that (see \cite[Section 4.1]{CKLW2015})
\begin{equation}   \label{Omega_is_0}
a_{(n)}\Omega=0, \quad a\in V, n\geq 0 ,
\end{equation}
we have that $\Omega$ is a primary vector in $V_0$. Furthermore, it is possible to prove that $\nu\in V_2$ and that it is a quasi-primary vector which is not primary if $c\not=0$ (see \cite[Section 4.1, 4.2]{CKLW2015}). 

To define a unitary structure on a VOA $V$, we recall some facts and definitions, which the reader may find in more detail in \cite[Section 4.3 and 5.1]{CKLW2015} and \cite[Section 5.2]{FHL1993}.
First of all, we say that a bilinear form $(\cdot, \cdot)$ on $V$ is \textit{invariant} if
\begin{equation*}
(Y(a,z)b, c)=(b,Y(e^{zL_1}(-z^{-2})^{L_0}a, z^{-1})c)
\hspace{0.5cm} 
a,b,c\in V.
\end{equation*}
In particular $(a,L_nb)=(L_{-n}a,b)$ for all $n\in\mathbb{Z}$.
Consequently, we say that a scalar product $(\cdot|\cdot)$ on $V$ (a positive-definite sesquilinear form, linear in the second variable) is \textit{invariant} if there exists an antilinear automorphisms $\theta$ of $V$, called a \textit{PCT operator} for $V$, which makes $(\theta(\cdot)|\cdot)$ an invariant bilinear form. Moreover, a scalar product on $V$ is said to be \textit{normalized} if $(\Omega|\Omega)=1$. 
We point out (see \cite[Proposition 5.1]{CKLW2015}) that for every normalized invariant scalar product on a VOA $V$, there exists a unique PCT operator.

We can now define a \textit{unitary vertex operator algebra} as a VOA $V$ equipped with an invariant normalized scalar product $(\cdot|\cdot)$. In this case, we say that a linear automorphism $g$ is \textit{unitary} if $(g(a)|g(b))=(a|b)$ for all $a,b\in V$, and we define $\mathrm{Aut}_{(\cdot|\cdot)}V$ as the group of unitary linear automorphisms; it is a compact subgroup of $\mathrm{Aut}V$ (see \cite[Lemma 5.20]{CKLW2015}). Similarly, one defines a \textit{unitary isomorphism} between two unitary VOAs.

An \textit{ideal} of a vertex algebra $V$ is a $T$-invariant vector subspace $\mathcal{I}\subseteq V$ such that $a_{(n)}b\in\mathcal{I}$ for all $a\in V$, $b\in\mathcal{I}$ and $n\in\Z$. By \cite[(4.3.1)]{Kac1997}, we also have $b_{(n)}a\in\mathcal{I}$. A vertex algebra $V$ is \textit{simple} if $V$ and $\left\{0\right\}$ are the only ideals.
Furthermore, by \cite[Proposition 5.3]{CKLW2015}, a unitary VOA $V$ is simple if and only if $V_0=\mathbb{C}\Omega$. It follows that every simple unitary VOA is of CFT type (cf. \cite[Remark 4.5]{CKLW2015}).

A \textit{vertex subalgebra} of a vertex algebra $V$ is a vector subspace $W$ of $V$ such that $\Omega\in W$ and $a_{(n)}b\in W$ for all $a,b\in W$. By \cite[p. 24]{CKLW2015}, every vertex subalgebra $W$ is $T$-invariant, thus $W$ is a vertex algebra too. Then we have the following definition: 

\begin{defin}  \label{def:unitary_subalgebra}
A \textit{unitary vertex subalgebra} of a unitary VOA $V$ with PCT operator $\theta$ is a vertex subalgebra $W$ of $V$ such that $\theta(W)\subseteq W$ and $L_1(W)\subseteq W$. For convenience we call a unitary vertex subalgebra simply a \textit{unitary subalgebra}.
\end{defin}

For convenience we have chosen to define unitary subalgebras by Definition \ref{def:unitary_subalgebra} instead of the standard one \cite[Definition 5.22]{CKLW2015}. In any case, these are equivalent thanks to \cite[Proposition 5.23]{CKLW2015}. Note also that, by \cite[Definition 5.22]{CKLW2015}, $L_0 (W)\subseteq W$ for any unitary subalgebra $W$.

Let $S$ be a subset of a unitary VOA $V$. We indicate with $W(S)$ the smallest vertex subalgebra of $V$ containing $S$ and we say that $S$ \textit{generates} a vertex subalgebra of $V$. Moreover, if $W(S)$ is a unitary subalgebra as in Definition \ref{def:unitary_subalgebra}, then we say that $S$ \textit{generates} a unitary subalgebra of $V$.

For further use, we state the following results about unitary subalgebras from \cite[Proposition 5.29]{CKLW2015}.
\begin{prop}  \label{prop:cklw_unitary_subvoa}
Let $(V, \Omega, T, Y, \nu, (\cdot|\cdot))$ be a simple unitary VOA and $W$ be a unitary subalgebra of $V$. Consider the projection $e_W$ into $W$ and set $\omega:=e_W(\nu)$ with $Y(\omega, z)=\sum_{n\in\mathbb{Z}}L_n^\omega z^{-n-2}$ the corresponding field. Then
\begin{itemize}
\item[(i)] $\theta(\omega)=\omega$.
\item[(ii)] $\omega$ is a Virasoro vector of $V$ and $L_n^\omega|_W=L_n|_W$ for $n\in\left\{-1,0,1\right\}$.
\item[(iii)] $\omega$ is a conformal vector for $W$ and $(W, \Omega, T, Y, \omega, (\cdot|\cdot))$ is a simple unitary VOA with PCT operator $\theta|_W$.
\end{itemize}
\end{prop}

Now, some classical examples of unitary VOAs.

\begin{ex}  \label{ex:trivial_subalgebra}
For every unitary VOA $V$, the vector subspace $\C\Omega$ is a unitary subalgebra, which we call the \textit{trivial subalgebra}. Recall that $T\Omega=0$ and $Y(\Omega,z)=I_V$. Furthermore, $e_{\C\Omega}(\nu)=0$, $(a\Omega|b\Omega)=\overline{a}b$ and $\theta(a\Omega)=\overline{a}\Omega$ for all $a,b\in\C$.
\end{ex}

\begin{ex}\label{ex:orbifold} (See \cite[Example 5.25]{CKLW2015}).
Let $G\subseteq\mathrm{Aut}_{(\cdot|\cdot)}V$ be a closed subgroup. Then
\begin{equation*}
V^G:=\left\{v\in V \barspace
g(v)=v 
\hspace{0.2cm} \forall g\in G
\right. \right\}
\end{equation*}
is a unitary subalgebra of $V$, called \textit{fixed point subalgebra}. Moreover, if $G$ is finite, $V^G$ is called \textit{orbifold subalgebra}. 
\end{ex}

\begin{ex}  \label{ex:virasoro_subalgebra}
By \cite[Example 5.24]{CKLW2015}, the conformal vector $\nu$ with central charge $c$ of a unitary VOA always generates a simple unitary subalgebra, called the \emph{Virasoro subalgebra} of $V$, indicated with $L(c,0)\subset V$. The admissible values of $c\in\C$ are the following ones:
$$
\begin{array}{lr}
 c\geq1, &  \\
c_m=1-\frac{6}{m(m+1)} \, , &
m\geq2.
\end{array}
$$
Moreover, $L(c,0)$ is a unitary VOA called the \emph{unitary Virasoro VOA with central charge} $c$, which can be directly constructed from the Virasoro Lie algebra of central charge $c$, namely the Lie algebra with formal generators $L_n$ and $1$ and commutation relations \eqref{eq:VirasoroCRs}, cf. \cite[Example 4.10]{Kac1997}, \cite[Section 6.1]{LepLi2012} and \cite[Lecture 3, 12.3]{KacRaina1987} for the construction and \cite[Section 4.1]{DonLin2014}, \cite[Example 5.6]{CKLW2015} for the unitary structure. The irreducible modules of $L(c,0)$ are denoted by $L(c,h)$ and arise from unitary irreducible positive-energy representations of the Virasoro Lie algebra with central charge $c$; $h$ denotes the lowest eigenvalue of $L_0$ on that module. The admissible values of $c$ and $h$, which completely determine each unitary irreducible positive-energy representation, are:
\[
\begin{array}{r@{\quad \textrm{and} \quad}l}
c\geq1 & h\geq0 \\
c_m=1-\frac{6}{m(m+1)} &
h_{p,q}(m)=\frac{((m+1)p-mq)^2-1}{4m(m+1)}
\end{array}
\]
for all $m\geq2$ and $1\leq q\leq p\leq m-1$, cf. also \cite[Lecture 8.4]{KacRaina1987}.
\end{ex}

\begin{ex}  \label{ex:family_theta_inv}
By \cite[Example 5.26]{CKLW2015}, we have that a family of $\theta$-invariant quasi-primary vectors in a unitary VOA $V$ generates a unitary subalgebra. Thus, if $W$ is a unitary subalgebra of $V$, the conformal vector $\omega=e_W(\nu)$ as in Proposition \ref{prop:cklw_unitary_subvoa} generates a unitary subalgebra of $V$ unitarily isomorphic to a unitary Virasoro VOA $L(c,0)$ for some allowed value of $c$ as in Example \ref{ex:virasoro_subalgebra}.
\end{ex}

\end{subsection}


\begin{subsection}{Rank-one lattice type VOAs} \label{section:rank_one_lattice}

We recall from \cite[Section 2]{DongGriess1998}, \cite[Section 4]{DonLin2014}, \cite[Section 7.1]{FLM1989}, \cite[Section 5.4]{Kac1997} and \cite[Chapter 6]{LepLi2012} some facts about rank-one lattice type VOAs. Let $L$ be an even rank-one positive definite lattice and let $\bilinear$ be the corresponding non-degenerate symmetric bilinear form. We set $\mathfrak{h}:=\mathbb{C}\otimes_\mathbb{Z} L$ and indicate the extension of $\bilinear$ by bilinearity to $\mathfrak{h}$ with the same symbol. Then there exist $N\in\Zplus$ and $J\in\mathfrak{h}$, which we call \textit{current vector}, such that $(J|J)=1$ and $L\cong L_{2N}:=\mathbb{Z}\sqrt{2N}J$. Let 
$$
\widehat{\mathfrak{h}}:=\mathfrak{h}\otimes\mathbb{C}[t,t^{-1}]\oplus\mathbb{C}K
$$  
be the affine Lie algebra with generators $J_n:= J\otimes t^n$, for $n\in\Z$, and central element $K$, with the commutation relations 
\begin{equation} \label{cr_heisenberg_algebra}
\left[J_m,J_n\right]=m \delta_{m,-n}K.
\end{equation}
$\widehat{\mathfrak{h}}$ is also called \textit{Heisenberg} or \textit{oscillator algebra}. For $\alpha\in\mathfrak{h}$ and $n\in\Z$ we write $\alpha_n:= \alpha\otimes t^n\in\widehat{\mathfrak{h}}$. Let $U(\widehat{\mathfrak{h}})$ be the universal enveloping algebra of $\widehat{\mathfrak{h}}$ (see \cite[Chapter V]{Jac1979}). A linear basis for $U\left(\widehat{\mathfrak{h}}\right)$ is given by
\begin{equation}  \label{basis_elem_uea}
J_{n_1}\cdots J_{n_s} K^k 
\end{equation}
for $s,k\in\Zpluseq, n_i\in\Z$ such that $n_i\leq n_j$ if $i<j$.
We define the irreducible $\widehat{\mathfrak{h}}$-module
$$
M(1):= U\left(\widehat{\mathfrak{h}}\right)/\mathcal{I}
$$
where $\mathcal{I}$ is the ideal in $U\left(\widehat{\mathfrak{h}}\right)$ generated by
$$
\{(K-1), J_n : n\in\Zpluseq\} .
$$
Then from \eqref{basis_elem_uea} we deduce that $M(1)$ is linearly generated by the following classes of elements:
$$
[1]=[K^k],\; \left[J_{n_1}\cdots J_{n_s}\right]=\left[J_{n_1}\cdots J_{n_s}K^k\right], \quad n_1\leq\cdots\leq n_s< 0,\; k,s\in\Zpluseq,
$$
which we simply indicate with
\begin{equation} \label{basis_m1}
\Omega,\; J_{n_1}\cdots J_{n_s}\Omega 
\end{equation}
respectively.

Define the vector space
$$
V_{L_{2N}}:=M(1)\otimes_\mathbb{C} \mathbb{C}[L_{2N}],
$$
where $\mathbb{C}[L_{2N}]$ is the group algebra, which we describe as linearly generated by formal elements $e^\alpha$, for $\alpha\in L_{2N}$ and multiplication given by
\begin{equation} \label{eq:group_alg_mult}
e^\alpha e^\beta= e^{\alpha+\beta}.
\end{equation}
In particular $1:=e^0$ is the identity.  
Note that a linear set of generators for $V_{L_{2N}}$ is given by
\begin{equation}  \label{eq:generators_vl2n}
\left\{v\otimes e^\alpha \, |\,  v\in M(1), \, \alpha\in L_{2N}\right\}.
\end{equation}
Note also that for $v$ as in \eqref{basis_m1}, \eqref{eq:generators_vl2n} is a linear basis for $V_{L_{2N}}$.

Thus $V_{L_{2N}}$ has a VOA structure, which we sum up in the following theorem. For the proof see \cite[Theorem 5.5, Proposition 5.5]{Kac1997} and \cite[Theorem 6.5.3]{LepLi2012}. Note that in the above references, the construction of the VOA $V_L$ associated to an even lattice $L$ involves the twisted group algebra $\mathbb{C}_\epsilon[L]$ instead of the group algebra defined above. In $\mathbb{C}_\epsilon[L]$, the multiplication \eqref{eq:group_alg_mult} is twisted by a 2-cocycle $\epsilon$. In any case, it turns out that the resulting VOA structure is independent of the choice of the 2-cocycle $\epsilon$ for the twisting. Thus we are allowed to use the trivial 2-cocycle, which is an admissible choice for the case of even rank-one lattices $L_{2N}$.

\begin{theo}    \label{theo:structure_vu12n}
For all $N\in\Zplus$, $V_{L_{2N}}$ has a structure of simple VOA of CFT type with vacuum vector $\Omega\otimes 1$ and the following data:
\begin{itemize}
\item[(i)] The state-field correspondence is given by
\begin{eqnarray}
Y(\alpha_{-1}\Omega\otimes 1, z)&=&\alpha(z):=\sum_{j\in\mathbb{Z}} \alpha_j z^{-j-1}  \label{field1}\\
Y(\Omega\otimes e^\alpha, z) &=& e_\alpha z^{\alpha_0} E_+(\alpha, z) E_-(\alpha, z)  \label{field2}\\
E_+(\alpha, z) &:=& \exp\left(-\sum_{j<0} \frac{\alpha_j}{j}z^{-j}\right) \label{fieldE+}\\
E_-(\alpha, z) &:=& \exp\left(-\sum_{j>0} \frac{\alpha_j}{j}z^{-j}\right) \label{fieldE-} \\
Y(J_{n_1}\cdots J_{n_s}\Omega\otimes e^\alpha, z) &=&
:\prod_{j=1}^s(\partial^{(-n_j-1)}J(z))
Y(\Omega\otimes e^\alpha, z):  \, ,
\label{field3_product}
\end{eqnarray}
where $\alpha\in\mathfrak{h}$ in \eqref{field1}, whereas $\alpha\in L_{2N}$ in \eqref{field2}-\eqref{field3_product}. In \eqref{field3_product}, $:\cdot:$ indicates the normally-ordered product of fields. In formulas \eqref{field1} - \eqref{fieldE-}, $e_\alpha$ is the operator on $V_{L_{2N}}$ of left multiplication by $\Omega\otimes e^\alpha$, whereas operators $\alpha_j$ and $z^{\alpha_0}$ act on $v\otimes e^\beta\in V_{L_{2N}}$ in the following manner
\begin{eqnarray}   
\alpha_0(v\otimes e^\beta)&=& (\alpha|\beta)v\otimes e^\beta 
\label{formal_series_ext_0}.\\ 
\alpha_j(v\otimes e^\beta)&=& (\alpha_jv)\otimes e^\beta  \hspace{0.5cm} \forall j\in\mathbb{Z}\setminus\{0\}
\label{formal_series_ext_j}. \\
z^{\alpha_0}\left(v\otimes e^\beta\right) &=& z^{(\alpha|\beta)}v\otimes e^\beta 
 \label{action_z_alpha0}.
\end{eqnarray}
 
\item[(ii)] The conformal vector with central charge 1 of $V_{L_{2N}}$ is 
\begin{equation}   \label{conformal_vector}
\nu\otimes 1:= \frac{1}{2} \left(J_{-1}J_{-1}\Omega\otimes 1 \right) \, .
\end{equation}
Thus, the corresponding Virasoro subalgebra is $L(1,0)$. The energy-momentum field $$Y(\nu\otimes 1, z)=\sum_{m\in\mathbb{Z}}L_mz^{-m-2}$$ is defined through formulas
\begin{align}
L_0(J_{n_1}\cdots J_{n_s}\Omega\otimes e^\alpha)=&
\left(\frac{(\alpha|\alpha)}{2}-\sum_{j=1}^s n_j \right)
J_{n_1}\cdots J_{n_s}\Omega\otimes e^\alpha \label{vir_op_0}\\
L_m =&\frac{1}{2}\sum_{j\in\mathbb{Z}}J_j J_{m-j},\quad  \forall m\in\mathbb{Z}\setminus\{0\} \label{vir_op_m}
\end{align}
and it satisfies the commutation relations
\begin{align}
\left[L_m,J_n\right]=&-nJ_{m+n} \hspace{0.5cm}\forall m,n\in\Z   \label{cr_vir_cur} \\
\left[L_m, L_n\right]=& (m-n)L_{m+n}+ \frac{m^3-m}{12}\delta_{m,-n} \hspace{0.5cm}\forall m,n\in\Z \,.\label{cr_vir}
\end{align}
In particular, $L_{-2}\left(\Omega\otimes 1\right)=\nu\otimes 1$. Furthermore, the vectors $J_{-1}\Omega\otimes 1$ and $\Omega\otimes e^\alpha$ are primary of conformal weights 1 and $\frac{(\alpha|\alpha)}{2}$ respectively.
\end{itemize}
\end{theo}

From \cite[(7.1.42), (3.2.17) and (2.1.11)]{FLM1989}, we have that
\begin{eqnarray}
E_+(\alpha, z) &= &
I+\sum_{n=1}^{+\infty}\frac{\left(-\sum_{j<0} \frac{\alpha_j}{j}z^{-j}\right)^n}{n!} 
\label{e+_expansion}\\
\left(-\sum_{j< 0} \frac{\alpha_j}{j}z^{-j}\right)^n &=&
\sum_{m=n}^{\infty}
\left(\sum_{\substack{j_1+\cdots +j_n=m \\ j_k>0}} \frac{\alpha_{-j_1}\cdots\alpha_{-j_n}}{j_1\cdots j_n}\right) z^m 
\label{e+_moltiplication}\\
E_-(\alpha, z) &=&
I+\sum_{n=1}^{+\infty}\frac{\left(-\sum_{j>0} \frac{\alpha_j}{j}z^{-j}\right)^n}{n!} 
\label{e-_expansion}\\
\left(-\sum_{j> 0} \frac{\alpha_j}{j}z^{-j}\right)^n &=&
(-1)^n\sum_{m=n}^{\infty}
\left(\sum_{\substack{j_1+\cdots +j_n=m \\ j_k>0}} \frac{\alpha_{j_1}\cdots\alpha_{j_n}}{j_1\cdots j_n}\right) z^{-m}
\label{e-_multiplication}\, .
\end{eqnarray}

\begin{rem}   \label{rem:cr_e+-}
Note that the operators $e_\alpha$ and $z^{\alpha_0}$ commute with $E_\pm(\alpha, z)$ (see also 
\cite[Proposition 6.4.5]{LepLi2012}).
\end{rem}

From equations \eqref{formal_series_ext_0}, \eqref{formal_series_ext_j}, \eqref{vir_op_0} and \eqref{vir_op_m} we obtain
\begin{eqnarray}
L_0(v\otimes 1) &=& \left(\sum_{j\in\Zplus}J_{-j}J_jv\right)\otimes 1 \\
L_m(v\otimes 1) &=& \left(\frac{1}{2}\sum_{j\in\mathbb{Z}}J_jJ_{m-j}v\right)\otimes 1, \,\, \forall m\in\Z\backslash\left\{0\right\}
\end{eqnarray}
for all $v\in M(1)$, so we write with an abuse of notation $(L_mv)\otimes 1$ or just $L_mv$ instead of $L_m(v\otimes 1)$ for all $m\in\Z$. 

Define an anti-linear involution $\theta$ on $V_{L_{2N}}$ which acts in the following manner on basis elements
\begin{equation} \label{theta_ext}
J_{n_1}\cdots J_{n_s}\Omega\otimes e^\alpha\longmapsto (-1)^s J_{n_1}\cdots J_{n_s}\Omega\otimes e^{-\alpha}.
\end{equation}
By \cite[Theorem 4.12]{DonLin2014}, $V_{L_{2N}}$ has a structure of unitary VOA with PCT operator $\theta$. Note also that all unitary structures on a simple VOA are equivalent up to unitary isomorphism (see \cite[p. 38]{CKLW2015}).

In the following, we list some well-known unitary subalgebras of $V_{L_{2N}}$ and we describe actions of some unitary automorphisms of $V_{L_{2N}}$ and the respective fixed point subalgebras.

A unitary subalgebra of $V_{L_{2N}}$ is $M(1)\otimes_{\mathbb{C}}\mathbb{C}1$ (which we indicate just with $M(1)$) with vacuum vector $\Omega\otimes 1$, conformal vector $\nu\otimes 1$ and PCT operator $\theta|_{M(1)}$. More in general, this is a well-know unitary VOA with central charge 1, called the \emph{Heisenberg} or \emph{current VOA} (see \cite[Section 4.3]{DonLin2014}).

Let $\phi$  be the linear automorphism on $V_{L_{2N}}$ defined by the following action on the basis elements:
\begin{equation}   \label{phi_voas}
J_{n_1}\cdots J_{n_s}\Omega\otimes e^\alpha\longmapsto (-1)^s J_{n_1}\cdots J_{n_s}\Omega\otimes e^{-\alpha}
\end{equation}
(not to be confused with the anti-linear PCT operator $\theta$ in \eqref{theta_ext}). Note that $\phi$ restricts to a linear automorphism on $M(1)$. Then we define the orbifold subalgebras $V_{L_{2N}}^+$ and $M(1)^+$ as fixed points of $V_{L_{2N}}$ and $M(1)$ respectively, with respect to $\phi$ (see also \cite[Section 2]{DongGriess1998}). 

We identify every $t\in\T:=\R/2\pi\Z$ with its representative in $[0,2\pi)$ and for any positive integer $N$, we define an action on $V_{L_{2N}}$ by automorphisms (see \cite[Section 2.3]{DonNag1998})
\begin{equation}  \label{action_torus}
g_{2N,t}:=\exp\left(i\frac{t}{\sqrt{2N}}J_0\right)=
\sum_{n\geq0}\frac{\left(i\frac{t}{\sqrt{2N}}J_0\right)^n}{n!} \,
\end{equation}
which act on generators of $V_{L_{2N}}$ in the following manner
\begin{equation}   \label{action_torus_basis_elem}
v\otimes e^\alpha\longmapsto e^{i\frac{t}{\sqrt{2N}}(J|\alpha)} v\otimes e^{\alpha}.
\end{equation}
Note that $J_0$ leaves invariant the finite-dimensional eigenspaces of $L_0$ so that  $\exp\left(i\frac{t}{\sqrt{2N}}J_0\right)$ 
is a well-defined vector space endomorphism of $V_{L_{2N}}$. Note also that for $\alpha = \sqrt{2N} J \in L_{2N}$ we have 
$e^{i\frac{t}{\sqrt{2N}}(J|\alpha)} = e^{it}$ so that $g_{2N,t}$ is not the identity when $t \neq 0$, $t\in [0,2\pi)$. 
We include the following well-known facts with proofs for the reader's convenience.

\begin{prop} \label{prop:d_infty_inner_aut}
For any positive integer $N$, the automorphisms $g_{2N,t}$ and $\phi$ defined above give an
embedding of
$D_\infty = \T\rtimes \Z_2$ into
$\mathrm{Aut}_{(\cdot| \cdot)}V_{L_{2N}}$. Moreover,
\begin{eqnarray}
V_{L_{2N}}^\T &=& M(1), 
\label{fixed_point_t}\\
V_{L_{2N}}^{D_\infty} &=& M(1)^+ 
\label{fixed_point_d_infty}.
\end{eqnarray}
\end{prop}

\begin{proof}
For every fixed $N$, let $G_{2N}\cong\T$ and $H\cong\Z_2$ be the groups generated by the automorphisms $g_{2N,t}$ for all $t\in\T$ and $\phi$ respectively. Using \eqref{phi_voas} and \eqref{action_torus_basis_elem}, it is easy to verify that $\phi g_{2N,t}\phi=g_{2N,t}^{-1}$ for all $g_{2N,t}\in G_{2N}$. It follows that the group generated by $G_{2N}$ and $H$ is isomorphic to $D_\infty$.

Now we prove that the automorphisms $g_{2N,t}$ and $\phi$ are unitary. Recall from \cite[Remark 5.18]{CKLW2015} that if $V$ is a simple unitary VOA then an automorphism $g$ of $V$ is unitary if and only if $g$ commutes with the PCT operator. The above condition is easily verified on basis elements of $V_{L_{2N}}$. Thus, $D_\infty$ embeds into $\mathrm{Aut}_{(\cdot|\cdot)}V_{L_{2N}}$.

Note that \eqref{fixed_point_d_infty} follows from \eqref{fixed_point_t} by definition. To prove \eqref{fixed_point_t}, consider $a\in V_{L_{2N}}^\T$ and its linear decomposition into basis elements, that is
\begin{equation}   \label{dec_in_basis_elem}
a=\sum_{j=1}^M a_jv_j\otimes e^{\alpha_j},
\end{equation}
with $\C\ni a_j\not=0$ for all $j$, $v_j$ one of the vectors in \eqref{basis_m1}. Applying a generic automorphism $g_{2N,t}$ to both sides of \eqref{dec_in_basis_elem}, we obtain
\begin{equation}  \label{cond_linear_indip}
\sum_{j=1}^M a_jv_j\otimes e^{\alpha_j}=a=
g_{2N,t}(a)=
\sum_{j=1}^M e^{i\frac{t}{\sqrt{2N}}(J|\alpha_j)}a_jv_j\otimes e^{\alpha_j}.
\end{equation}
By linear independence of $v_j\otimes e^{\alpha_j}$, \eqref{cond_linear_indip} is satisfied if and only if $e^{i\frac{t}{\sqrt{2N}}(J|\alpha_j)}=1$ for all $j$ and $t\in\T$, i.e., $\alpha_j=0$ for all $j$. To sum up, $a\in V_{L_{2N}}^\T$ if and only if $\alpha_j=0$ for all $j$, which is equivalent to $a\in M(1)$. This completes the proof.
\end{proof}

In the following, for every fixed $N$, we identify $\T$ with the closed subgroup $G_{2N}$ of $\mathrm{Aut}_{(\cdot|\cdot)}V_{L_{2N}}$ as in the proof of Proposition \ref{prop:d_infty_inner_aut}. We also identify $\Z_k\subset\T$ with the cyclic subgroup of $\mathrm{Aut}_{(\cdot|\cdot)}V_{L_{2N}}$ of order $k$ generated by automorphisms $g_{N,\frac{2m\pi}{k}}$ for $m\in\left\{0,\dots,k-1\right\}$, and $D_k:=\Z_k\rtimes\Z_2$ with the corresponding dihedral subgroup of $\mathrm{Aut}_{(\cdot|\cdot)}V_{L_{2N}}$ of order $k$.

Note that $L_{2Nk^2}$ is a sublattice of $L_{2N}$ for all positive integers $N$ and $k$. This inclusion induces an isomorphism of vector spaces
\begin{equation}  \label{inclusion_rank_one_latt}
\iota: V_{L_{2Nk^2}}\longleftrightarrow V:=\bigoplus_{\alpha\in L_{2Nk^2}}M(1)\otimes_\C \C e^\alpha \subset V_{L_{2N}} \, .
\end{equation}
$V$ is clearly a vertex subalgebra of $V_{L_{2N}}$ thanks to Theorem \ref{theo:structure_vu12n}. Note that $\nu\in V$ and that the PCT operator $\theta$ is independent of $N$ and $k$. Thus $V$ is a unitary subalgebra of $V_{L_{2N}}$ as in Definition \ref{def:unitary_subalgebra}. It follows that \eqref{inclusion_rank_one_latt} is a unitary isomorphism of unitary VOAs. In particular, $\iota(M(1))=M(1)$. Therefore, from now onwards, we consider $V_{L_{2Nk^2}}$ as a unitary subalgebra of $V_{L_{2N}}$.

\begin{prop}   \label{prop:fixed_point_zk}
For every positive integer $N$ and $k$, we have that:
\begin{eqnarray}
V_{L_{2N}}^{\Z_k} &=& V_{L_{2Nk^2}}, 
\label{fixed_point_zk}\\
V_{L_{2N}}^{D_k} &=& V_{L_{2Nk^2}}^+ 
\label{fixed_point_dk}.
\end{eqnarray} 
\end{prop}

\begin{proof}
First note that \eqref{fixed_point_dk} follows from \eqref{fixed_point_zk} by definition. Then it is sufficient to prove that $V_{L_{2N}}^{\Z_k}$ is equal to $V$ as in \eqref{inclusion_rank_one_latt} to conclude. Using the same argument as in the proof of Proposition \ref{prop:d_infty_inner_aut}, we have that $a\in V_{L_{2N}}^{\Z_k}$ if and only if 
$e^{i\frac{2\pi}{k\sqrt{2N}}(J|\alpha_j)}=1$ for all $j$. $\alpha_j\in\Z\sqrt{2N}J$ for all $j$, thus the former condition is satisfied if and only if $\alpha_j\in\Z\sqrt{2N}kJ$ for all $j$. Then $V=V_{L_{2N}}^{\Z_k}$ as desired. 
\end{proof}

\begin{rem}   \label{rem:action_d_k^t}
The closed subgroups of $D_\infty\subset \mathrm{Aut}_{(\cdot|\cdot)}V_{L_{2N}}$ are the circle group $\T$, cyclic groups $\Z_k$, dihedral groups $D_k$ and their conjugates $D_k^t:=g_{2N,t}D_kg_{2N,-t}$ for all $t\in\T$. Note that $D_k^t=D_k^{t+\pi}$ for all $k\in\Zplus$ and $t\in\T$ ($g_{2N,\pi}$ is in the center of  $D_\infty$). Furthermore obviously $D_k=D_k^t$ for all $k\in\Zplus$ and $t\in\Z_k$. Thus, it is easy to verify that $V_{L_{2N}}^{D_k^t}=g_{2N,t}\left(V_{L_{2N}}^{D_k}\right)$ for all $k\in\Zplus$ and $t\in\T$, which is unitarily isomorphic to $V_{L_{2N}}^{D_k}$ for all $k\in\Zplus$ and $t\in\T$.
\end{rem}

The VOAs just mentioned above have a decomposition into Virasoro modules as described in \cite[Section 2]{DongGriess1998}. In particular, if $N$ is not a perfect square (the square of a positive integer) then we have the following decompositions into irreducible Virasoro modules.
\begin{align}
V_{L_{2N}} = &
\bigoplus_{p\geq 0} L(1,p^2) \oplus \bigoplus_{m>0}2L(1,Nm^2)
\label{dec_vir_2N}\\
M(1)= &
\bigoplus_{p\geq 0} L(1,p^2)    \label{dec_vir_cur}\\
V_{L_{2N}}^+ = &
\bigoplus_{p\geq 0} L(1,4p^2) \oplus \bigoplus_{m>0}L(1,Nm^2)
\label{dec_vir_2N+} \\
M(1)^+ =&
\bigoplus_{p\geq 0} L(1,4p^2)  
\label{dec_vir+}\, .
\end{align} 

\begin{rem}\label{rem:primary_vect}
As stated in \cite[Section 2]{DongGriess1998}, the decompositions in \eqref{dec_vir_2N} and \eqref{dec_vir_cur} are due to the fact that the following isomorphisms of Virasoro modules hold
\begin{equation}
M(1)\otimes_{\mathbb{C}} \mathbb{C}e^{\pm \sqrt{2N}mJ} \cong
L(1,Nm^2), \quad \forall m\in\Zplus \, .
\end{equation}   
Thus, for all $m>0$, $\Omega\otimes e^{\pm \sqrt{2N}mJ}$ form a basis for the subspace of $V_{L_{2N}}$ of primary vectors of conformal weight $Nm^2$. As a matter of fact, it is straightforward to check that the former elements are homogeneous of conformal weight $Nm^2$ for all $m>0$ by \eqref{vir_op_0} and that $L_n(\Omega\otimes e^{\pm \sqrt{2N}mJ})=0$ for all $n>0$ and $m>0$ by \eqref{vir_op_m}. Thus, they generate (as Virasoro modules) the two copies of $L(1,Nm^2)$. This also implies that every $L(1,Nm^2)$ in \eqref{dec_vir_2N+} is generated by the $\phi$-invariant primary vector $\Omega\otimes(e^{\sqrt{2N}mJ}+e^{-\sqrt{2N}mJ})$ of conformal weight $Nm^2$. 
\end{rem}
\end{subsection}
\end{section}


\begin{section}{The classification in the VOA setting}  \label{section:compl_class}

In this section we classify all unitary subalgebra of $V_{L_{2N}}$, for all $N\in\Zplus$. To do this we state the following Galois correspondence for unitary VOAs which is crucial for the classification result. This is a variant of \cite[Theorem 7.7]{CKLW2015} which can be entirely formulated and proved in the unitary VOA setting thanks to \cite[Theorem 3]{DonMas1999}.

\begin{theo} \label{theo:galois_correspondence_voa}
Let $V$ be a simple unitary VOA and $G$ a closed subgroup of 
$\mathrm{Aut}_{(\cdot| \cdot)}V$ which is topologically isomorphic to a Lie group. Then the map $H\mapsto V^H$ gives a one-to-one
correspondence between the closed subgroups $H$ of $G$ and the unitary subalgebras $ W \subset V$ containing $V^G$.
\end{theo} 

\begin{proof}
We first show that every unitary subalgebra $W \subset V$ which contains $V^G$ is an orthogonally complemented subVOA in the sense of \cite[Definition 2]{DonMas1999}. To this end it is rather straightforward to see that  in our case the subspace in Eq. (1.3) in  \cite[Definition 2]{DonMas1999} coincides with the orthogonal complement 
$$W^\perp :=  \{a \in V
\;|\; (a|b) = 0\; \textrm{for all}\; b \in W  \}$$
which is easily seen to be  a $W$-submodule of $V$ because of the invariance property of the scalar product and the fact that $W$ is a unitary subalgebra. Now, for every closed subgroup $H$ of $G$, $V^H$ is a unitary subalgebra of $V$ by Example \ref{ex:orbifold}. Moreover, since
$\mathrm{Aut}_{(\cdot| \cdot)} V$ is compact by \cite[Lemma 5.20]{CKLW2015}, $G$ must be topologically isomorphic to a compact Lie group. Then the result follows from \cite[Theorem 3]{DonMas1999}.  
\end{proof}

The unitary subalgebras of $V_{L_2}$ have already been classified in \cite[Theorem 8.13]{CKLW2015}. Here, we give a proof in the unitary VOA setting using the Galois correspondence stated in Theorem \ref{theo:galois_correspondence_voa} instead of \cite[Theorem 7.7]{CKLW2015}.

First of all, recall that $V_{L_2}$ is identified (cf. \cite[Example 8.8]{CKLW2015}) with the simple unitary VOA $V_{\mathfrak{sl}(2,\mathbb{C})_1}$ built from the affine Lie algebras $\mathfrak{sl}(2,\mathbb{C})_1$ associated to the Lie group $\SU$ (see \cite[Chapter 7]{Kac1995}, \cite{GooWal1984} and \cite{PreSeg1986}). Consider the ortonormal generators $J^a$ for $a\in\left\{x,y,z\right\}$ of the complex Lie algebra $\mathfrak{sl}(2,\mathbb{C})$ and corresponding elements $J_0^a$ of the associated loop algebra. Then we have a group of unitary isomorphisms of $V_{\mathfrak{sl}(2,\mathbb{C})_1}$, isomorphic to $\SO$ and generated by operators of the form $\exp\left(itJ_0^a\right)$ for all $t\in\T$ and all $a\in\left\{x,y,z\right\}$. Then, $\SO\subseteq\operatorname{Aut}_{(\cdot|\cdot)}V_{L_2}$ (see \cite[p. 63]{CKLW2015}). Remember that the closed subgroups of $\SO$ are, up to isomorphism, the circle group $\mathbb{T}$, all cyclic groups $\mathbb{Z}_k$, the infinite dihedral group $D_\infty$, all finite dihedral groups $D_k$ and the three platonic groups $E_m$. Finally, note that the action of $D_\infty$ introduced in Proposition \ref{prop:d_infty_inner_aut} can be embedded in $\SO$ as described above, identifying the current vector $J$ with $\sqrt{2}J^z$, cf. \cite[Section 5B]{BMT1988}. Therefore, we have the following.

\begin{theo}    \label{theo:subvoas_vl2}
All unitary subalgebras of $V_{L_2}\cong V_{\mathfrak{sl}(2,\mathbb{C})_1}$ are given by the family of fixed point subalgebras $V_{L_2}^H$ for every closed subgroup $H\subseteq \SO$ together with the trivial subalgebra $\C\Omega\otimes 1$. In particular, $V_{L_2}^{\SO} = L(1,0)$ and consequently $\operatorname{Aut}_{(\cdot|\cdot)}V_{L_2}=\SO$.
\end{theo}

\begin{proof}
Let $W$ be a nontrivial subalgebra. Then by \cite[Lemma 8.10]{CKLW2015}, $L(1,0)\subset W$. Moreover, by \cite[Proposition 8.11]{CKLW2015} (cf. also \cite[Corollary 2.4]{DongGriess1998})
$L(1,0) = V_{L_2}^{SO(3)}$ and the claim follows from Theorem \ref{theo:galois_correspondence_voa}. 
\end{proof}

For references about VOAs built from affine Lie algebras, cf.~\cite[Example 4.9a and Section 5.7]{Kac1997} and \cite[Section 6.2]{LepLi2012}, \cite[Section 7.2]{FLM1989}, and \cite[Section 4.2]{DonLin2014} for the unitary structure.

For $N$ different from a perfect square, the classification of unitary subalgebras of $V_{L_{2N}}$ will rely on Theorem \ref{theo:m1+_in_subvoa}, the key ingredient of our classification result, which states that any unitary subalgebra properly containing $L(1,0)$ also contains $M(1)^+$. To prove it, we begin with some preliminary results.

\begin{lem} \label{lem:primary_vector_4}
Let $N\in\Zplus$.
\begin{itemize}
\item[(i)] The vector 
$$
u\otimes 1:=\left(\frac{1}{2}(J_{-1})^4\Omega -J_{-3}J_{-1}\Omega +\frac{3}{4}(J_{-2})^2\Omega\right)\otimes 1
$$
is primary of conformal weight 4. In particular, if $N$ is not a perfect square, then it is the highest weight vector of the unique irreducible Virasoro submodule $L(1,4)$ of $V_{L_{2N}}$ as in \eqref{dec_vir_2N}.
\item[(ii)] We also have
\begin{align*}
L_{-2}(\nu\otimes 1) =& \left(\frac{1}{4}(J_{-1})^4\Omega +J_{-3}J_{-1}\Omega \right)\otimes 1  \\
L_{-4}(\Omega \otimes 1) =& \left(\frac{1}{2}(J_{-2})^2\Omega +J_{-3}J_{-1}\Omega \right)\otimes 1 \, .
\end{align*}
\end{itemize}
\end{lem}

\begin{proof}
(i) From formula \eqref{vir_op_0} it is clear that $u\otimes1$ is homogeneous of conformal weight 4. According to \eqref{dec_vir_2N}, if $u\otimes 1$ is primary and $N$ is not a perfect square, it will be the highest weight vector of the unique irreducible Virasoro submodules $L(1,4)$ of $V_{L_{2N}}$. We need to prove that $L_mu=0$ for all $m\geq1$. Using equation \eqref{cr_vir_cur}, we see that
\begin{align}   
\frac{1}{2}L_m(J_{-1})^4\Omega =& \frac{1}{2}\Big( J_{m-1}(J_{-1})^3\Omega +J_{-1}J_{m-1}(J_{-1})^2\Omega \nonumber\\
&+(J_{-1})^2J_{m-1}J_{-1}\Omega + 
(J_{-1})^3J_{m-1}\Omega +(J_{-1})^4L_m\Omega \Big) \label{j4}\\
-L_mJ_{-3}J_{-1}\Omega =&
-\left(3J_{m-3}J_{-1}\Omega +J_{-3}J_{m-1}\Omega +J_{-3}J_{-1}L_m\Omega \right)\label{j31}\\
\frac{3}{4}L_m(J_{-2})^2\Omega=& \frac{3}{4}\left(2J_{m-2}J_{-2}\Omega + 2J_{-2}J_{m-2}\Omega + (J_{-2})^2L_m\Omega\right) .\label{j2}
\end{align}
Note that the last term on the right-hand side of each of equations \eqref{j4} - \eqref{j2} is 0, because $L_m\Omega=0$ for all $m\geq1$. Moreover, by the commutation relations \eqref{cr_heisenberg_algebra} we know $J_j$ and $J_k$ commute if $k\not = -j$. This means that for all $m\geq5$ every term in equations \eqref{j4} - \eqref{j2} is zero since $J_{j}\Omega=0$ for all $j\geq 0$. Therefore, it remains to study the four cases $m\in\{1,2,3,4\}$ one by one. For $m=1$ we have
\begin{align*}
\frac{1}{2}L_1(J_{-1})^4\Omega =& 2(J_{-1})^3J_0\Omega =0 \\
-L_1J_{-3}J_{-1}\Omega =& -3J_{-2}J_{-1}\Omega -J_{-3}J_0\Omega
=-3J_{-2}J_{-1}\Omega \\
\frac{3}{4}L_1(J_{-2})^2\Omega =& \frac{3}{4}\left(2J_{-1}J_{-2}\Omega + 2J_{-2}J_{-1}\Omega\right)=3J_{-2}J_{-1}\Omega
\end{align*}
and the sum is clearly 0.
For $m=2$,
\begin{align*}
\frac{1}{2}L_2(J_{-1})^4\Omega=
&\frac{1}{2}\left(\right. J_1(J_{-1})^3\Omega +J_{-1}J_1(J_{-1})^2\Omega +\\ 
&+(J_{-1})^2J_1J_{-1}\Omega + (J_{-1})^3J_1\Omega\left.\right) \\
=&\frac{1}{2}\left((3+2+1)(J_{-1})^2\Omega+4(J_{-1})^3J_1\Omega\right) \\
=&3(J_{-1})^2\Omega\\
-L_2J_{-3}J_{-1}\Omega =& -3J_{-1}J_{-1}\Omega -J_{-3}J_1\Omega
=-3(J_{-1})^2\Omega \\
\frac{3}{4}L_2(J_{-2})^2\Omega =&
\frac{3}{4}\left(2J_0J_{-2}\Omega + 2J_{-2}J_0\Omega\right)= 0
\end{align*}
and again the sum is 0. For $m=3$,
\begin{align*}
\frac{1}{2}L_3(J_{-1})^4\Omega =& 2(J_{-1})^3J_2\Omega =0 \\
-L_3J_{-3}J_{-1}\Omega =& -3J_0J_{-1}\Omega -J_{-3}J_2\Omega=0 \\
\frac{3}{4}L_3(J_{-2})^2\Omega =&
\frac{3}{4}\left(2J_1J_{-2}\Omega + 2J_{-2}J_1\Omega\right)=0 \, .
\end{align*}
Finally for $m=4$,
\begin{align*}
\frac{1}{2}L_4(J_{-1})^4\Omega =& 2(J_{-1})^3J_3\Omega =0 \\
-L_4J_{-3}J_{-1}\Omega =& -3J_1J_{-1}\Omega -J_{-3}J_3\Omega
=-3\Omega\\
\frac{3}{4}L_4(J_{-2})^2\Omega =&
\frac{3}{4}\left(2J_2J_{-2}\Omega + 2J_{-2}J_2\Omega \right)
= 3\Omega
\end{align*}
which adds up to 0 as well.

(ii) From (ii) of Theorem \ref{theo:structure_vu12n} we have that
\begin{equation*}
\begin{split}
L_{-2}\nu=\frac{1}{2}L_{-2}(J_{-1})^2\Omega &=
\frac{1}{2}\left(2J_{-3}J_{-1}\Omega +(J_{-1})^2L_{-2}\Omega\right) \\
&= J_{-3}J_{-1}\Omega +\frac{1}{4}(J_{-1})^4\Omega
\end{split}
\end{equation*}
and using \eqref{vir_op_m}
\begin{equation*}
\begin{split}
L_{-4}\Omega 
=\frac{1}{2}\sum_{j\in\mathbb{Z}} J_jJ_{-4-j}\Omega
=\frac{1}{2}\sum_{j=-3}^{-1} J_jJ_{-4-j}\Omega &=
\frac{1}{2}\left(2J_{-3}J_{-1}\Omega+(J_{-2})^2\Omega\right)\\
&=J_{-3}J_{-1}\Omega+\frac{1}{2}(J_{-2})^2\Omega \, .
\end{split}
\end{equation*}
\end{proof}

\begin{prop}   \label{prop:cur_primary4}
Let $N,m$ be positive integers such that $g:=\sqrt{2N}m\geq2$. Let $b\in\mathbb{C}\setminus\{0\}$ and set
$$
e^g_b:= \Omega\otimes\left(e^{gJ}+be^{-gJ}\right) \, .
$$
Then 
\begin{align*}
\left(\Omega\otimes e^{\pm gJ}\right)_{\left(g^2-2\right)}
\left(\Omega\otimes e^{\mp gJ}\right) =& 
(\pm gJ_{-1}\Omega)\otimes 1 \\
(e^g_b)_{\left(g^2-5\right)}e^g_b =& bv_g\otimes 1,
\end{align*}
where
$$
v_g:=\frac{g^4}{12}(J_{-1})^4\Omega 
+\frac{2g^2}{3}J_{-3}J_{-1}\Omega
+\frac{g^2}{4}(J_{-2})^2\Omega 
\, .
$$
\end{prop}

\begin{proof}
Note that the $z^{-g^2+4}$-coefficient of the formal $z$-series $Y(e^g_b, z)e^g_b$  corresponds to
$(e^g_b)_{\left(g^2-5\right)}e^g_b$. We have that
\begin{align*}
Y(e^g_b, z)e^g_b =& A+bB \\
A :=& Y(\Omega\otimes e^{gJ},z)(\Omega\otimes e^{gJ})+b^2Y(\Omega\otimes e^{-gJ},z)(\Omega\otimes e^{-gJ})\\
B :=& Y(\Omega\otimes e^{gJ},z)(\Omega\otimes e^{-gJ})+Y(\Omega\otimes e^{-gJ},z)(\Omega\otimes e^{gJ}).
\end{align*}
From formulas \eqref{e-_expansion} - \eqref{e-_multiplication} we deduce that $E_-(cgJ,z)(\Omega\otimes e^{dgJ})$ is equal to
\begin{equation}  \label{calculation_e-}
\left(\Omega+\sum_{n=1}^{+\infty} \frac{(-1)^n}{n!}\left( \sum_{m=n}^{\infty}
\left(\sum_{\substack{j_1+\cdots +j_n=m \\ j_k>0}} \frac{(cg)^nJ_{j_1}\cdots J_{j_n}\Omega}{j_1\cdots j_n}\right) z^{-m} \right)\right)
\otimes e^{dgJ},
\end{equation}
for all $c, d\in\{-1,+1\}$. Formula \eqref{calculation_e-} is equal to $\Omega\otimes e^{dgJ} $ because $J_j\Omega=0$ for all $j\geq0$ by construction.
From Remark \ref{rem:cr_e+-} and \eqref{field2}, we have that 
\begin{eqnarray}
Y(\Omega\otimes e^{\pm gJ},z)(\Omega\otimes e^{\pm gJ})&=&
z^{g^2}\left(E_+(\pm gJ,z)\Omega\right)\otimes e^{\pm 2gJ}
\label{prod_dir1}\\
Y(\Omega\otimes e^{\pm gJ},z)(\Omega\otimes e^{\mp gJ})&=&
z^{-g^2}\left(E_+(\pm gJ,z)\Omega\right)\otimes 1
\label{prod_dir1.1}
\end{eqnarray}
From \eqref{e+_moltiplication} we deduce that the lowest $z$-power in equation \eqref{prod_dir1} is $g^2>0$. Thus we can restrict our attention just to $B$ and therefore to equation \eqref{prod_dir1.1} because $-g^2+4\leq0$. Then we have to consider the $z^4$-coefficient of $E_+(\pm gJ,z)\Omega$ in equation \eqref{prod_dir1.1}. This means that in \eqref{e+_expansion} - \eqref{e+_moltiplication}, we can restrict the calculation to the cases $m=4$ and $n\in\{1,2,3,4\}$. There we find
\begin{equation*}
\begin{array}{r@{:\quad}l}
n=1&   \frac{\pm g}{4}J_{-4} 
 \\
 n=2&   \frac{1}{2}\left(
\frac{(\pm g)^2}{4}(J_{-2})^2+\frac{(\pm g)^2}{3}J_{-3}J_{-1}
+\frac{(\pm g)^2}{3}J_{-1}J_{-3}
\right)
\\
n=3&  \frac{1}{3!}\left(
 \frac{(\pm g)^3}{2}J_{-2}(J_{-1})^2+
 \frac{(\pm g)^3}{2}J_{-1}J_{-2}J_{-1}+
 \frac{(\pm g)^3}{2}(J_{-1})^2J_{-2}
\right)
\\
n=4&  \frac{1}{4!}\left((\pm g)^4(J_{-1})^4  \right)
\, .
\end{array}
\end{equation*} 
Using the commutation relations \eqref{cr_heisenberg_algebra}, we obtain as $z^4$-coefficient of $E_+(\pm gJ,z)\Omega$ in equation \eqref{prod_dir1.1}, the element
\begin{equation}  \label{prod_dir4}
\pm \frac{g}{4}J_{-4}\Omega 
+\frac{g^2}{8}(J_{-2})^2\Omega
+\frac{g^2}{3}J_{-3}J_{-1}\Omega
\pm \frac{g^3}{4}J_{-2}(J_{-1})^2\Omega
+\frac{g^4}{24}(J_{-1})^4  \Omega
\, .
\end{equation}
In order to compute $B$, we have to add up the two versions of \eqref{prod_dir4}, so the summands with $\pm$ cancel while the other ones double, and we obtain for the $z^{-g^2+4}$-coefficient of $Y(e^g_b,z)e^g_b$:
$$
bv_g\otimes 1= b
\left(\frac{g^4}{12}(J_{-1})^4\Omega 
+\frac{2g^2}{3}J_{-3}J_{-1}\Omega
+\frac{g^2}{4}(J_{-2})^2\Omega \right)\otimes 1 \, .
$$

In the same manner, the $z^{-g^2+1}$-coefficient of the formal $z$-series 
$Y(\Omega\otimes e^{\pm gJ}, z)\left(\Omega\otimes e^{\mp gJ}\right)$ corresponds to $\left(\Omega\otimes e^{\pm gJ}\right)_{\left(g^2-2\right)}
\left(\Omega\otimes e^{\mp gJ}\right)$ respectively. Thus we have to consider the $z$-coefficient of $E_+(\pm gJ,z)\Omega$ in equation \eqref{prod_dir1.1}. This corresponds to fixing $m=1=n$ in formulas \eqref{e+_expansion} - \eqref{e+_moltiplication} to obtain $\pm gJ_{-1}$. Thus $(\pm gJ_{-1}\Omega)\otimes1$ are the desired $z^{-g^2+1}$-coefficients.
\end{proof}

\begin{prop}   \label{prop:linearly_generated}
For every nonzero complex number $g$, the primary vector $u\otimes 1\in V_{L_{2N}}$ of conformal weight 4 as in Lemma \ref{lem:primary_vector_4} is a linear combination of vectors $L_{-4}\Omega\otimes 1$, $L_{-2}\nu\otimes 1$ and $v_g\otimes 1$ as in Proposition \ref{prop:cur_primary4}.
\end{prop}
\begin{proof}
It is easy to see that the three vectors form a basis for the vector subspace of $V_{L_{2N}}$ generated by $(J_{-1})^4\Omega\otimes1$, $J_{-3}J_{-1}\Omega\otimes1$ and $(J_{-2})^2\Omega\otimes1$. It follows that $u\otimes 1$ must be a linear combination of these three vectors.
\end{proof}

From previous results we obtain a generalization of \cite[Theorem 2.9]{DongGriess1998}.

\begin{cor} \label{cor:dong_general}
For any integer $N$ which is not a perfect square, $V_{L_{2N}}^+$ is generated (as VOA) by the conformal vector $\nu\otimes 1$ and the primary vector $e^{\sqrt{2N}}_1=\Omega\otimes \left(e^{\sqrt{2N}J}+e^{-\sqrt{2N}J}\right)$.
\end{cor}

\begin{rem}   \label{rem:notation_dong_griess}
Note that in \cite{DongGriess1998}, the authors indicate with $\omega$ and $\beta$ our conformal vector $\nu\otimes 1$ and the current vector $J$ respectively (see \cite[p. 264]{DongGriess1998}). Moreover, their $u^m$ stands for the highest weight vector of the irreducible Virasoro submodule $L(1,m^2)$ of $M(1)$ (see \cite[p. 268]{DongGriess1998}), thus their $u^2$ is our $u\otimes1$ as in Lemma \ref{lem:primary_vector_4}. Finally, $e^n$ in \cite[p. 269]{DongGriess1998} coincides with our $e^{\sqrt{2N}}_1$, identifying $n$ with $N$ when they are different from a perfect square.
\end{rem}

\begin{proof}[Proof of Corollary \ref{cor:dong_general}]
Due to Remark \ref{rem:primary_vect} the vector $\Omega\otimes \left(e^{\sqrt{2N}J}+e^{-\sqrt{2N}J}\right)$ is in $V_{L_{2N}}^+$. Using Proposition \ref{prop:cur_primary4}, the vector $v_g\otimes 1\in V_{L_{2N}}^+$. The result then follows from Proposition \ref{prop:linearly_generated} and \cite[Theorem 2.9]{DongGriess1998}. 
\end{proof}

We are now ready to prove the key ingredient of our classification proof.
\begin{theo}     \label{theo:m1+_in_subvoa}
Let $N$ be an integer which is not a perfect square. Every unitary subalgebra $W$ of $V_{L_{2N}}$ properly containing $L(1,0)$ contains also $M(1)^+$.
\end{theo}
\begin{proof}
Consider the decomposition of $V_{L_{2N}}$ as in \eqref{dec_vir_2N}. Keeping in mind the notational correspondence given in Remark \ref{rem:notation_dong_griess}, if $W$ contains $L(1,4p^2)$ for at least one $p>0$, then the theorem follows from \cite[Theorem 2.7(2)]{DongGriess1998}. Similarly, if $W$ contains $L(1,p^2)$ for at least one odd $p>0$, then we have that $\bigoplus_{\substack{j=0 \\ \mathrm{even}}}^{2p} L(1,j^2)\subseteq W$ by \cite[Lemma 2.6]{DongGriess1998}. Then, the theorem follows from \cite[Theorem 2.7(2)]{DongGriess1998}. We therefore have to prove that $W$ contains at least one $L(1,p^2)$ with $p>0$. 

Suppose for contradiction that $W$ does not contain any $L(1,p^2)$ with $p>0$, i.e., $W\cap L(1,p^2)=\{0\}$. This implies that $W$ has the following decomposition into Virasoro modules
\begin{equation}  \label{dec_W}
W= L(1,0)\oplus \bigoplus_{m>0} a_m L(1, Nm^2)
\end{equation}
where $a_m\in\left\{0,1,2\right\}$. 

We first want to prove that $a_m\not =2$, for all $m$. If $a_m$ were equal to $2$ for some $m$ then by Remark \ref{rem:primary_vect}, $\Omega\otimes e^{\pm \sqrt{2N}mJ}\in W$. By Proposition \ref{prop:cur_primary4}, $J_{-1}\Omega\otimes 1$ would then lie in $W$ but according to Theorem \ref{theo:structure_vu12n}(ii), we have $J_{-1}\Omega\otimes 1\in L(1,1)$, which cannot be the case as $W\cap L(1,1)=\{0\}$ by assumption. Therefore, $a_m\in\left\{0,1\right\}$; moreover, at least one $a_m$ equals $1$ because by assumption $W\not= L(1,0)$. Fix such an $m$.

Second, from Remark \ref{rem:primary_vect}, we know that, for this $m$, there exists a primary vector of conformal weight $g^2/2:=Nm^2$ in $W$ which must be a linear combination of the form $\Omega\otimes\left(ae^{gJ}+be^{-gJ}\right)$, for some $a,b\in\mathbb{C}$. $W$ is a unitary subalgebra, thus by Definition \ref{def:unitary_subalgebra}, $W$ must be invariant under the PCT operator $\theta$. If $a$ were $0$ then $\theta\left(\Omega\otimes be^{-gJ}\right)=\Omega\otimes\overline{b}e^{gJ}$ would be in $W$, which means that $a_m=2$; however, since $a_m=1$, we find that $a\not=0$. Similarly, $b\not =0$; thus, up to rescaling, we can suppose $a=1$. With $e^g_b=\Omega\otimes\left(e^{gJ}+be^{-gJ}\right)\in W$ as in Proposition  \ref{prop:cur_primary4}, the vector $v\otimes 1$ must lie in $W$. Furthermore, $L_{-2}\nu\otimes 1$ and $L_{-4}\Omega\otimes 1$ are in $W$ because they are vectors of $L(1,0)$. Thus by Proposition \ref{prop:linearly_generated}, $u\otimes 1$ lies in $W$. On the other hand, $u\otimes 1$ is the primary vector of conformal weight $4$ generating $L(1,4)$, so we obtain $W\cap L(1,4)\not=\{0\}$, which leads to a contradiction. Therefore $W$ must contain at least one $L(1,p^2)$ with $p>0$, which concludes the proof of the theorem.
\end{proof}

Theorem \ref{theo:m1+_in_subvoa} also allows us to explicitly calculate $\mathrm{Aut}_{(\cdot|\cdot)}V_{L_{2N}}$ for $N$ not a perfect square. We highlight that $\mathrm{Aut}V_{L_{2N}}$ for all $N$ has been calculated by \cite[Theorem 2.1]{DonNag1998}.

\begin{cor}    \label{cor:calculation_aut}
For any integer $N$ which is not a perfect square, we have that
\begin{equation}
\mathrm{Aut}_{\bilinear}V_{L_{2N}}=D_\infty .
\end{equation}
\end{cor}

\begin{proof}
We have $D_\infty\subseteq\mathrm{Aut}_{\bilinear} V_{L_{2N}}$ by Proposition \ref{prop:d_infty_inner_aut}. Suppose for contradiction there exists $g\in\mathrm{Aut}_{\bilinear}V_{L_{2N}}\setminus D_\infty$ and let $G$ be the (proper) closed abelian subgroup of $\mathrm{Aut}_{\bilinear}V_{L_{2N}}$ generated by $g$. By the Galois correspondence given in Theorem \ref{theo:galois_correspondence_voa} and the fact that $L(1,0)$ is the only proper unitary subalgebra of $M(1)^+=V_{L_{2N}}^{D_\infty}$ as proved in \cite[Corollary 2.8]{DongGriess1998}, we can deduce that $V_{L_{2N}}^{\mathrm{Aut}_{\bilinear}V_{L_{2N}}}=L(1,0)$. Then, using again the Galois correspondence, $V_{L_{2N}}^G$ is a unitary subalgebra of $V_{L_{2N}}$ properly containing $L(1,0)$. By Theorem \ref{theo:m1+_in_subvoa}, $V_{L_{2N}}^G$ contains also $M(1)^+=V_{L_{2N}}^{D_\infty}$, so $G$ is a subgroup of $D_\infty$ by the Galois correspondence, which is a contradiction, so $\mathrm{Aut}_{\bilinear}V_{L_{2N}} = D_\infty$.
\end{proof}

From formula \eqref{vir_op_0}, we deduce that for every $N\in\Zplus$, the vector subspace of $V_{L_{2N}}$ of vectors of conformal weight 2 is spanned by
\begin{equation}  \label{vect_vl2n_conf_weight_2}
\begin{aligned}
J_{-2}\Omega\otimes 1,\;
\nu\otimes 1,\; J_{-1}\Omega\otimes e^{\sqrt{2}J}, \;
J_{-1}\Omega\otimes e^{-\sqrt{2}J}
&\quad (N=1) \\
J_{-2}\Omega\otimes 1,\;
\nu\otimes 1,\; \Omega\otimes e^{2J}, \;
\Omega\otimes e^{-2J}
&\quad (N=2) \\
J_{-2}\Omega\otimes 1, \;
\nu\otimes 1 \;
&\quad (N>2)
\, .
\end{aligned}
\end{equation}
Moreover, it is easy to verify by \eqref{vir_op_m} that
\begin{equation} \label{quasi_prim_vl4}
\begin{aligned}
L_1 \left( J_{-1}\Omega\otimes e^{\pm\sqrt{2}J} \right) =& 
\pm\sqrt{2}\Omega\otimes e^{\pm\sqrt{2}J}, \\
L_1 \left( J_{-2}\Omega\otimes 1 \right)
=& 2J_{-1}\Omega\otimes 1, \\
L_1(\nu\otimes1) =&0, \\
L_1\left(\Omega\otimes e^{\pm 2J}\right) =&0.
\end{aligned}
\end{equation}
Now, let $W$ be a unitary subalgebra of $V_{L_{2N}}$. Consider the conformal vector $\omega=e_W(\nu\otimes 1)$ of $W$ as in Proposition \ref{prop:cklw_unitary_subvoa} and let $Y(\omega,z):=\sum_{n\in\mathbb{Z}}L_n^{\omega}z^{-n-2}$ be the corresponding energy-momentum field. Using (ii) of Proposition \ref{prop:cklw_unitary_subvoa} together with \cite[Theorem 4.10(iv)]{Kac1997} and the quasi-primarity of the conformal vector (see Subsection \ref{subsec:unitary_voas}), we have that
\begin{equation}
\begin{aligned}
L_0 \omega = & L_0^\omega \omega=2\omega
\label{omega_conformal} \\
L_1 \omega = & L_1^\omega \omega=0
\, .
\end{aligned}
\end{equation}
Thus, equations \eqref{omega_conformal} imply that $\omega$ must be a quasi-primary vector in $V_{L_{2N}}$ of conformal weight 2. Hence, according to \eqref{vect_vl2n_conf_weight_2} and \eqref{quasi_prim_vl4}, recalling that all nonzero vectors on the right in \eqref{quasi_prim_vl4} are linearly independent, $\omega$ must be equal to a multiple of $\nu\otimes1$ for $N\not=2$. Then we have proved the following result (cf. also \cite[Proposition 5.1]{Carpi1999}):

\begin{prop}   \label{prop:every_sub_contains_l10}
Let $N\neq 2$. Then, every nontrivial unitary subalgebra
of $V_{L_{2N}}$ contains $L(1,0)$.
\end{prop} 

\begin{rem}
For $N=k^2$ with $k$ a positive integer, Proposition \ref{prop:every_sub_contains_l10} follows also directly from Proposition \ref{prop:fixed_point_zk} and Theorem \ref{theo:subvoas_vl2}.
\end{rem}

In the following we investigate the case $N=2$ to be able to complete the classification. 

First of all, we know from \cite[Lemma 3.1]{DGH1998} (put $2J=\alpha$ there), cf. also the proof of \cite[Theorem 6.3]{DMZ1994} putting $2J=\beta$ there, that $V_{L_4}$ contains at least two distinct copies of $L\left(\frac{1}{2},0\right)$ generated by Virasoro vectors
\begin{eqnarray}
\omega_0 &:=& \frac{\nu\otimes1}{2}+\frac{\Omega\otimes\left( e^{2J}+e^{-2J}\right)}{4} \\
\omega_\pi &:=& \frac{\nu\otimes1}{2}-\frac{\Omega\otimes\left( e^{2J}+e^{-2J}\right)}{4}
\, .
\end{eqnarray}
Let $W_0$ and $W_\pi$ be the vertex subalgebras of $V_{L_4}$ generated by $\omega_0$ and $\omega_\pi$ respectively. Then, thanks to Example \ref{ex:family_theta_inv}, they are unitary subalgebras of $V_{L_4}$, unitarily isomorphic to the unitary Virasoro VOA $L\left(\frac12,0\right)$. Therefore, our goal is to prove the following result.

\begin{theo}    \label{theo:case_n=2}
For every $t\in\T$, the vector 
\begin{equation}
\omega_t:=\frac{\nu\otimes1}{2}+\frac{\Omega\otimes\left( e^{it}e^{2J}+e^{-it}e^{-2J}\right)}{4} \in V_{L_4}
\end{equation}
is a Virasoro vector with central charge $\frac12$, generating a unitary subalgebra $W_t\subset V_{L_4}$ unitarily isomorphic to $L\left(\frac12,0\right)$.
Moreover, if $W$ is a nontrivial unitary subalgebra of $V_{L_4}$ that does not contain $L(1,0)$ then $W=W_t=g_{4,t}\left(W_0\right)$ for some $t\in\T$.
\end{theo}
To prove Theorem \ref{theo:case_n=2}, we need the following result.

\begin{lem}  \label{lem:omega_tau}
Let $W\subset V_{L_4}$ be a nontrivial unitary subalgebra which does not contain $L(1,0)$, then
$$
\omega=e_W(\nu\otimes1)\in\left\{\omega_t \;|\; t\in\T \right\}
$$
\end{lem}

\begin{proof}
First of all, $L(1,0)\not\subset W$ implies that the conformal vector $\omega$ of $W$ must be different from $\nu\otimes 1$. Moreover, it must be different from a multiple of $\nu\otimes 1$ because $W$ is nontrivial and because it has to satisfy $L_0^\omega \omega = 2 \omega$ as in \eqref{omega_conformal}. Consider now the energy-momentum field $Y(\omega,z):=\sum_{n\in\mathbb{Z}}L_n^{\omega}z^{-n-2}$ corresponding to $\omega$. Then \eqref{omega_conformal} implies that $\omega$ must be a linear combination of the three quasi-primary vectors of $V_{L_4}$ in \eqref{vect_vl2n_conf_weight_2}, that is
$$
\omega=a\nu\otimes 1+b\Omega\otimes e^{2J}+d\Omega\otimes e^{-2J}
$$
for some $a,b,d\in\mathbb{C}$. Note that either $b\not=0$ or $d\not=0$ because $\omega$ is not a multiple of $\nu\otimes 1$. Using (i) of Proposition \ref{prop:cklw_unitary_subvoa}, namely $\theta(\omega) = \omega$, we obtain
$$
\omega=a\nu\otimes1+\Omega\otimes\left(be^{2J}+\overline{b}e^{-2J}\right)
=a\nu\otimes1+be^2
$$  
for some $a\in\mathbb{R}$, $b\in\mathbb{C}\setminus\{0\}$ and $e^2:=e^2_{\frac{\overline{b}}{b}}=\Omega\otimes\left(e^{2J}+\frac{\overline{b}}{b}e^{-2J}\right)$.

We want to calculate $L_0^\omega\omega$. Consider
\begin{equation}   \label{omega_main_calculation}
\sum_{n\in\mathbb{Z}}L_n^\omega\omega z^{-n-2}=Y(\omega, z)\omega=aY(\nu\otimes 1, z)\omega+bY(e^2, z)\omega
\, .
\end{equation}
Then we have that
\begin{equation}    \label{l0_omega}
L_0^\omega \omega = aL_0\omega + abC_{-2,\nu}+b^2C_{-2,e^2}=2a\omega + abC_{-2,\nu}+b^2C_{-2,e^2}
\end{equation}
where $C_{-2,\nu}$ and $C_{-2, e^2}$ are $z^{-2}$-coefficients of $Y(e^2,z)(\nu\otimes1)$ and $Y(e^2,z)e^2$ respectively.

Proceeding as in the proof of Proposition \ref{prop:cur_primary4}, we find
\begin{equation*}
C_{-2, e^2} =\frac{\overline{b}}{b}4(J_{-1})^2\Omega\otimes 1 =\frac{\overline{b}}{b}8\nu\otimes 1  
\, .
\end{equation*}
To calculate $C_{-2,\nu}$ , consider 
\begin{align*}
Y(\Omega\otimes e^{2J},z)(\nu\otimes 1) =& E_+(2J,z)E_-(2J,z)\nu\otimes e^{2J} \\
\frac{\overline{b}}{b}Y(\Omega\otimes e^{-2J},z)(\nu\otimes 1) =&
\frac{\overline{b}}{b}E_+(-2J,z)E_-(-2J,z)\nu\otimes e^{-2J}
\end{align*}
where we have used firstly Remark \ref{rem:cr_e+-} and secondly formula \eqref{action_z_alpha0}. Using commutation relations \eqref{cr_heisenberg_algebra} and the fact that $J_j\Omega=0$ for all $j\geq0$ by construction, we get
\begin{align*}
J_j(\nu\otimes e^{\pm2J})=&\frac{1}{2}J_j\left(J_{-1}\right)^2\Omega\otimes e^{\pm2J}=0 
\quad\forall j\geq2 \\
\left(J_1\right)^j(\nu\otimes e^{\pm2J})=&\frac{1}{2}\left(J_1\right)^j\left(J_{-1}\right)^2\Omega\otimes e^{\pm2J}=0 
\quad\forall j\geq 3\\ 
\left(J_1\right)^2(\nu\otimes e^{\pm2J})=&\frac{1}{2}\left(J_1\right)^2\left(J_{-1}\right)^2\Omega\otimes e^{\pm2J}
=\Omega\otimes e^{\pm2J} \\
J_1(\nu\otimes e^{\pm2J})=&\frac{1}{2}J_1\left(J_{-1}\right)^2\Omega\otimes e^{\pm2J}=J_{-1}\Omega\otimes e^{\pm2J}
\, .
\end{align*}
Together with \eqref{e-_expansion} - \eqref{e-_multiplication} this implies 
\begin{equation*}
\begin{split}
E_-&(\pm 2J,z)(\nu\otimes e^{\pm2J}) =\\
&= \nu\otimes e^{\pm2J} +
\frac{-1}{1!}\cdot \frac{\pm 2J_1(\nu\otimes e^{\pm2J})}{1}z^{-1} 
+ \frac{(-1)^2}{2!} \cdot\frac{\left(\pm2J_1\right)^2(\nu\otimes e^{\pm2J})}{1\cdot1}z^{-2} \\
&= \nu\otimes e^{\pm2J} \mp\left(2J_{-1}\Omega\otimes e^{\pm2J} \right)z^{-1}+\left(2\Omega\otimes e^{\pm2J}\right) z^{-2}
\, .
\end{split}
\end{equation*}

Using \eqref{e+_expansion} and \eqref{e+_moltiplication}, we find
\begin{equation*}
E_+(\pm2J,z)=I\pm 2J_{-1}z+O(z^2),
\end{equation*}
thus
\begin{equation*}
\begin{split}
Y(\Omega\otimes e^{2J},z)(\nu\otimes 1) =& E_+(2J,z)E_-(2J,z)\nu\otimes e^{2J} \\
=& \left(2\Omega\otimes e^{2J}\right) z^{-2} +O(z^{-1})
\end{split}
\end{equation*}
and
\begin{equation*}
\begin{split}
\frac{\overline{b}}{b} Y(\Omega\otimes e^{-2J},z)(\nu\otimes 1) =& \frac{\overline{b}}{b}E_+(-2J,z)E_-(-2J,z)\nu\otimes e^{-2J}\\
=&\frac{\overline{b}}{b} \left(2\Omega\otimes e^{-2J}\right) z^{-2} + O(z^{-1}).
\end{split}
\end{equation*}
We deduce that
\begin{equation}
C_{-2,\nu}=2e^2  
\, .
\end{equation}
We can therefore rewrite identity \eqref{l0_omega}:
\begin{equation}
L_0^\omega\omega=2a\omega+2abe^2+8|b|^2\nu\otimes 1=2(a^2+4|b|^2)\nu\otimes1+4abe^2
\label{l0_omega_omega}
\end{equation}
Imposing equation \eqref{omega_conformal}, namely $L_0^\omega\omega = 2\omega$, and using \eqref{l0_omega_omega}, we have the following identity
\begin{equation*}
2a\nu\otimes 1+2be^2 = 2(a^2+4|b|^2)\nu\otimes1+4abe^2
\end{equation*}
which has precisely the solutions
\[
a = \frac{1}{2},\quad b = \frac{e^{it}}{4}, \quad t\in\T.
\]
Thus $\omega\in\{\omega_t:\;  t\in\T\}$.
\end{proof}

\begin{proof}[Proof of Theorem \ref{theo:case_n=2}]
First, note that $\omega_t=g_{4,t}(\omega_0)$  for all $t\in\T$, where $g_{4,t}$ as in \eqref{action_torus} is a unitary automorphism of $V_{L_4}$. Thus, $\omega_t\in V_{L_4}$ are still Virasoro vectors of central charge $\frac{1}{2}$ and by Example \ref{ex:family_theta_inv}, they generate unitary subalgebras $W_t=g_{4,t}\left(W_0\right)$ of $V_{L_4}$ unitarily isomorphic to $L\left(\frac12,0\right)$.

Second, we have to show that every unitary subalgebra $W\subset V_{L_4}$ is of the above type. Lemma \ref{lem:omega_tau} proves that $\omega=e_W(\nu\otimes1)$ must be of the form $\omega_t$, for some $t\in\T$. Accordingly, $W$ is a unitary vertex operator algebra with conformal vector $\omega_t$ and hence its central charge is $c=\frac12$ . On the other hand $L\left(\frac12,0\right)$ is the unique, up to isomorphism, unitary vertex operator algebra with central charge $c=\frac12$, see e.g. \cite{DongLin2015}. It follows that $W=W_t$.
\end{proof}

\begin{rem}   \label{rem:abouot_w_tau}
Note that every $W_t$ is contained in $V_{L_4}^{D_1^t}$ for all $t\in\T$. By Remark \ref{rem:action_d_k^t}, $D_1^t= D_1^{t+\pi}$ for all $t\in\T$. Moreover, under the identification $V_{L_4}^+=V_{L_4}^{D_1}=W_0\otimes W_\pi$ given by \cite[Lemma 3.1 (ii)]{DGH1998}, we have that $V_{L_4}^{D_1^t}=W_t\otimes W_{t+\pi}$ for all $t\in\T$.
\end{rem}

Then we can prove the classification theorem.

\begin{theo}   \label{theo:resume_theo}
The nontrivial unitary subalgebras $W$ of the
rank-one lattice VOAs $V_{L_{2N}}$ are classified as follows. Apart from the Virasoro subalgebra $L(1,0)$, we have:
\begin{itemize}
\item[(i)] If $N=k^2$ for some positive integer k then, after the
identification of $V_{L_{2N}}$ with
$V_{L_2}^{\Z_k}$, $W= V_{L_2}^{H}$ for some closed subgroup $H\subseteq SO(3)$ containing $\Z_k$.
\item[(ii)] If $N>2$ is not a perfect square then $W= V_{L_{2N}}^{H}$ for some closed subgroup $H\subseteq D_\infty$.
\item[(iii)] If $N=2$, then either $W= V_{L_4}^{H}$ for some closed subgroup $H\subseteq D_\infty$ or $W=W_t$ for some $t\in\T$.
\end{itemize}
\end{theo}

\begin{proof}
First of all, note that $L(1,0)\subset V_{L_{2N}}$ for all $N\in\Zplus$.
\begin{itemize}
\item[(i)] The case $k=1$ is Theorem \ref{theo:subvoas_vl2}. By Proposition \ref{prop:fixed_point_zk}, $V_{L_{2k^2}}=V_{L_2}^{\Z_k}$ for all $k\in\Zplus$ is a unitary subalgebra of $V_{L_2}$. Then the result follows from the Galois correspondence in Theorem \ref{theo:galois_correspondence_voa} and from Theorem \ref{theo:subvoas_vl2}.
\item[(ii)] By Proposition \ref{prop:every_sub_contains_l10}, every unitary subalgebra of $V_{L_{2N}}$ contains $L(1,0)$ and thus contains also $M(1)^+$ by Theorem \ref{theo:m1+_in_subvoa}. On the other hand, $M(1)^+=V_{L_{2N}}^{D_\infty}$ by Proposition \ref{prop:d_infty_inner_aut} and therefore the result follows from the Galois correspondence in Theorem \ref{theo:galois_correspondence_voa}.
\item[(iii)] If $L(1,0)\subset W$, then we can proceed as in (ii) above. The other case is covered by Theorem \ref{theo:case_n=2}.
\end{itemize}
\end{proof}
\end{section}

\begin{section}{The classification in the conformal net setting}    \label{section:conformal_nets}

Our aim in this section is to arrive at the analogue of Theorem \ref{theo:resume_theo} for conformal nets. For the sake of completeness, we start by recalling basic definitions about conformal nets. For more details, we refer to \cite[Section 1-3]{CKLW2015} and references therein.

Let $\mathcal{J}$ be the family of open, connected, non-empty and non-dense subsets (also called \textit{intervals}) of the unit circle $S^1$. For all $I\in\mathcal{J}$, we set $I':=S^1\setminus\overline{I}$.
A \textit{(irreducible) conformal net} is a family $\mathcal{A}=(\A(I))_{I\in\mathcal{J}}$ of von Neumann algebras on a fixed separable Hilbert space $\mathcal{H}$, which has the following properties:
\begin{itemize}
\item \textit{(Isotony)}. For every $I_1, I_2\in\mathcal{J}$, if $I_1\subseteq I_2$, then $\mathcal{A}(I_1)\subseteq \mathcal{A}(I_2)$.
\item \textit{(Locality)}. For every $I_1, I_2\in\mathcal{J}$, if $I_1\cap I_2=\emptyset$, then $[\mathcal{A}(I_1), \mathcal{A}(I_2)]=0$. 
\item \textit{(M\"obius covariance)}. Let $\mathrm{M\ddot{o}b} \cong \mathrm{PSL}(2,\mathbb{R})$ be the group of M\"obius transformations of the circle $S^1$. Then there exists a strongly continuous unitary representation $U$ of $\mathrm{M\ddot{o}b}$ on $\mathcal{H}$ such that
\begin{equation} 
U(\gamma)\mathcal{A}(I)U(\gamma)^{-1}=\mathcal{A}(\gamma I)
\end{equation}
for all $\gamma\in\mathrm{M\ddot{o}b}$ and $I\in\mathcal{J}$.
\item \textit{(Positivity of the energy)}. The infinitesimal generator $H$ of the one-parameter rotation subgroup of the representation $U$ (also called \textit{conformal Hamiltonian}) is a positive operator on $\mathcal{H}$.
\item \textit{(Existence and uniqueness of the vacuum)}. Let $\bigvee_{I\in\mathcal{J}}\mathcal{A}(I)$ be the von Neumann algebra generated by $\mathcal{A}(I)$ for all $I\in\mathcal{J}$. Then there exists a unique (up to a phase) $U$-invariant unit vector $\Omega\in\mathcal{H}$ such that $\Omega$ is cyclic for $\bigvee_{I\in\mathcal{J}}\mathcal{A}(I)$. $\Omega$ is called the \textit{vacuum vector} of the theory.
\item \textit{(Conformal covariance)}. Denote by $\mathrm{Diff}^+(S^1)$ the group of orientation-preserving diffeomorphisms of the unit circle. For all $I\in\mathcal{J}$, let $\mathrm{Diff}(I)\subset\mathrm{Diff}^+(S^1)$ be the subgroup of all $\gamma\in\mathrm{Diff}^+(S^1)$ such that $\gamma(z)=z$ for all $z\in I'$. Then, there exists a strongly continuous projective unitary representation of $\mathrm{Diff}^+(S^1)$ which extends $U$ and such that (we use the same symbol $U$ for the extension)
\begin{eqnarray}
U(\gamma)\mathcal{A}(I)U(\gamma)^{-1} &=&
\mathcal{A}(\gamma I) 
\hspace{0.5cm} \forall \gamma\in\mathrm{Diff}^+(S^1) \\
U(\gamma)AU(\gamma)^{-1} &=&
A \hspace{0.5cm} \forall A\in\mathcal{A}(I), \forall\gamma\in\mathrm{Diff}(I')
\, .
\end{eqnarray}
\end{itemize}
To be precise, a conformal net should actually be defined as a
triple $(\mathcal{A}, U, \Omega)$ satisfying the above properties. It can be shown that $U$ is completely determined by the pair $(\mathcal{A}, \Omega)$, see e.g. [10, Section 3.3] so that a conformal net may be identified with the pair  
$(\mathcal{A}, \Omega)$. However, we follow the usual convention to denote the conformal net simply by $\mathcal{A}$ keeping in mind all the relevant structures in the definition.

A \textit{conformal subnet} of $\mathcal{A}$ is a family $\mathcal{B}=(\mathcal{B}(I))_{I\in\mathcal{J}}$ of von Neumann algebras on $\mathcal{H}$ such that
\begin{itemize}
\item $\mathcal{B}(I)\subseteq \mathcal{A}(I)$ for all $I\in\mathcal{J}$;
\item if $I_1, I_2\in\mathcal{J}$ with $I_1\subseteq I_2$ then $\mathcal{B}(I_1)\subseteq \mathcal{B}(I_2)$;
\item $U(\gamma)\mathcal{B}(I)U(\gamma)^{-1}=\mathcal{B}(\gamma I)$ for all $\gamma\in\mathrm{M\ddot{o}b}$ and $I\in\mathcal{J}$.
\end{itemize}
Strictly speaking, $\mathcal{B}$ is not a conformal net because $\Omega$ is not a cyclic vector for $\mathcal{B}$ with respect to the Hilbert space $\mathcal{H}$. However, $\mathcal{B}$ becomes a conformal net after restricting to the closure $\mathcal{H}_\mathcal{B}$ of $\left(\bigvee_{I\in\mathcal{J}}\mathcal{B}(I)\right)\Omega$.

We define the group of \textit{automorphisms} of a net $\mathcal{A}$ as
\begin{equation}
\mathrm{Aut}\mathcal{A}:=
\left\{g\in B(\mathcal{H}) \barspace
g(\Omega)=\Omega,
\hspace{0.2cm}
g\mathcal{A}(I)g^{-1}=\mathcal{A}(I) 
\hspace{0.2cm}\forall I\in\mathcal{J} 
\right. \right\}.
\end{equation}
It follows from \cite{MTW2018} that $\mathrm{Aut}\mathcal{A}$ is a compact topological group with respect to the strong topology induced by $B(\mathcal{H})$. Let $G\subseteq\mathrm{Aut}\mathcal{A}$ be a compact subgroup. Then $\A^G$ defined by
\begin{equation}
\mathcal{A}^G(I):=\left\{A\in\mathcal{A}(I) \barspace
gAg^{-1}=A
\hspace{0.2cm} \forall g\in G
\right. \right\}.
\end{equation}
is a conformal subnet of $\A$ called a \textit{fixed point subnet}.
If $G$ is a finite group, one calls $\mathcal{A}^G$ an \textit{orbifold subnet}.

The core ingredient of the classification proof in this section is the relationship between conformal nets and unitary VOAs established in \cite{CKLW2015}. The idea there was to start with a simple unitary VOA $(V,\Omega,Y,T,\nu, (\cdot|\cdot))$ and to try to construct a conformal net $(\mathcal{A}_{V}(I))_{I\in\mathcal{J}}$ on the Hilbert space completion $\mathcal{H}$ of $V$ as follows:
\begin{equation}   \label{conf_net_from_voa}
\mathcal{A}_{V}(I):=W^*\left\{
Y(a,f) \barspace
a\in V, \hspace{0.2cm}
f\in C^\infty(S^1,\mathbb{R}), \hspace{0.2cm}
\mathrm{supp}f\subset I
\right. \right\} ,
\end{equation}
where $W^*\left\{\cdot\right\}$ indicates the smallest von Neumann algebra affiliated with a family of closed densely defined operators $Y(a,f)$ called \emph{smeared fields} (see \cite[Section 2.2]{CKLW2015}). Under an additional technical assumption on the VOA called \emph{strong locality}, all involved elements are well-defined and the construction works and does indeed define a conformal net, cf. \cite[Theorem 6.8]{CKLW2015}. Furthermore, by \cite[Theorem 6.9]{CKLW2015}, we have the following equality: 
\begin{equation}    \label{eq:correspondence_aut}
\mathrm{Aut}_{(\cdot|\cdot)}V=\mathrm{Aut}\A_V.
\end{equation}
More precisely, the extension by continuity of unitary automorphisms of $V$ to the Hilbert space completion $\mathcal{H}$
gives rise to a group isomorphism from $\mathrm{Aut}_{(\cdot|\cdot)}V$ onto $\mathrm{Aut}\A_V$.

Without the assumption of strong locality it is unclear whether or not one can define a conformal net in this way. Some of the key results from \cite{CKLW2015} we are going to use are the following, cf. \cite[Theorems 7.1, 7.5, Proposition 7.6]{CKLW2015}:
\begin{theo}   \label{theo:subnets_subvoas_corr}
Let $W$ be a unitary subalgebra of a strongly local simple unitary VOA $V$. Then $W$ is still strongly local and simple. Furthermore, there is a one-to-one correspondence between unitary subalgebra of $V$ and conformal subnets of $\A_V$ given by the map $W\mapsto \A_W$.
\end{theo}
\begin{prop}   \label{prop:corr_fixed_point}
Let $(V,(\cdot|\cdot))$ be a strongly local simple unitary VOA and $G$ be a closed subgroup of $\mathrm{Aut}_{(\cdot|\cdot)}V=\mathrm{Aut}\A_V$. Then $\A_V^G=\A_{V^G}$.
\end{prop}

Let us now look at the models we are investigating in this paper. We write $\A_{\Uone}$ for the \textit{$\Uone$-current (conformal) net}, and for every positive integer $N\geq1$, let $\A_{\Uone_{2N}}$ denote the extension conformal net of $\A_{\Uone}$ built from the current $J(z)$ as described in \cite[Proposition 4.1]{BMT1988}. We call it an \textit{even rank-one lattice conformal net}. We point out that such an extension is maximal if and only if $N$ is not a multiple of a perfect square. As a matter of fact, if $N=N'k^2$ for positive integers $N'$ and $k$, then $\A_{\Uone_{2N'k^2}}\subset\A_{\Uone_{2N'}}$ (cf. \cite[p. 37]{BMT1988}). 
We indicate with $\A_{\Vir,c}$ the \emph{Virasoro (conformal) net} with central charge $c\in\left\{\frac12,1\right\}$. Then \cite[Example 8.8]{CKLW2015} can be summarized as follows:

\begin{prop}  \label{prop:vln-u12n}
For every positive integer $N$, the rank-one lattice type VOA $V_{L_{2N}}$ is strongly local and $\A_{V_{L_{2N}}}=\A_{\Uone_{2N}}$.
\end{prop}

In the following, combining Proposition \ref{prop:vln-u12n}, Proposition \ref{prop:corr_fixed_point} and Theorem \ref{theo:subnets_subvoas_corr}, we can translate the results obtained for rank-one lattice type VOAs in Section \ref{section:rank_one_lattice}, \ref{section:compl_class}, especially Theorem \ref{theo:resume_theo}, to the conformal net setting.
 
First, by \eqref{eq:correspondence_aut} there is a well-defined action of $D_\infty$ on $\A_{\Uone_{2N}}$, for every $N\in\Zplus$. Thus we have that $\A_{\Uone_{2Nk^2}}=\A_{\Uone_{2N}}^{\Z_k}$ for all positive integers $N$ and $k$ thanks to \eqref{fixed_point_zk}. Furthermore, we have a well-defined action of $\SO$ on $\A_{\Uone_2}$ which extends the one of $D_\infty$. Set $\A_t:=\A_{W_t}\cong\A_{\Vir,\frac12}$ for all $t\in\T$ with $W_t$ as in Theorem \ref{theo:case_n=2}. Then, we can state the following classification result:
 
\begin{theo}    \label{theo:class_conf_net_sett}
The nontrivial conformal subnets $\A$ of the even rank-one lattice conformal nets $\A_{U(1)_{2N}}$ are classified as follows. Apart from the Virasoro conformal net $\A_{\Vir,1}$, we have:
\begin{itemize}
\item[(i)] If $N=k^2$ for some positive integer k then, after the identification $\A_{\Uone_{2k^2}}=\A_{\Uone_2}^{\Z_k}$, $\A= \A_{\Uone_2}^{H}$ for some closed subgroup $H\subseteq SO(3)$ containing $\Z_k$.
\item[(ii)] If $N>2$ is not a perfect square then $\A=\A_{U(1)_{2N}}^{H}$ for some closed subgroup $H\subseteq D_\infty$.
\item[(iii)] If $N=2$, then either $\A= \A_{U(1)_4}^{H}$ for some closed subgroup $H\subseteq D_\infty$ or $\A=\A_t$ for some $t\in\T$.  
\end{itemize}
\end{theo} 
 
Regarding the case $N=1$, recall that the classification has been known for a while, cf. \cite[Theorem 3.2]{Carpi1999} and \cite[Proposition 5]{Rehren1994}. Recall also that $\A_{\Uone_2}$ is isomorphic to the level 1 loop group net $\mathcal{A}_{\SU_1}$ (see \cite[Section 5B]{BMT1988}).

 Let $\mathcal{A}_{\Uone_{2N}}^+ := \mathcal{A}_{\Uone_{2N}}^{D_1}$ and $\mathcal{A}_{\Uone}^+ := \mathcal{A}_{\Uone}^{D_1}$, then we have the conformal net results corresponding to \eqref{fixed_point_t}, \eqref{fixed_point_d_infty}, \eqref{fixed_point_dk} and Remark \ref{rem:action_d_k^t}, namely
\begin{eqnarray}
\A_{\Uone_{2N}}^\T &=& \A_{\Uone} \, ,  \\
\A_{\Uone_{2N}}^{D_\infty} &=& \A_{\Uone}^+ \, ,  \\
\A_{\Uone_{2N}}^{D_k} &=& \A_{\Uone_{2Nk^2}}^+ \, , \\
\A_{\Uone_{2N}}^{D_k^t} &=& g_{2N,t}\A_{\Uone_{2Nk^2}}^+( g_{2N,t})^{-1}\, .
\end{eqnarray}

We can summarize the correspondence between some of the subtheories of rank-one lattice type models in the VOA and in the conformal net setting with the following diagram:
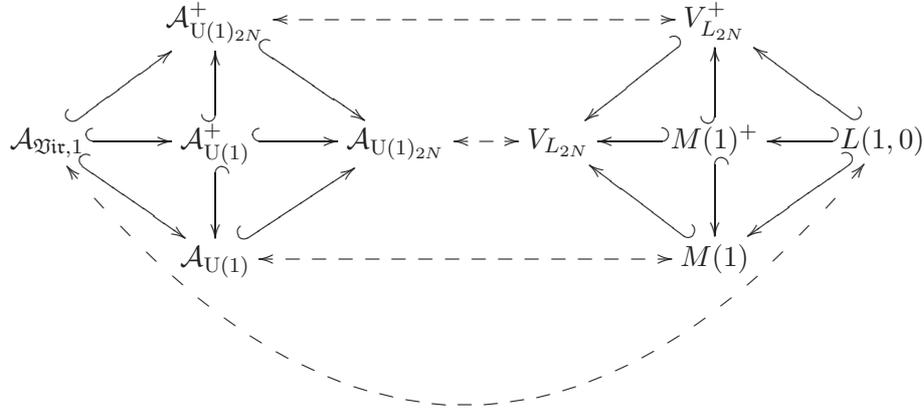
\begin{equation*}  
\xymatrix{
& \mathcal{A}_{\Uone_{2N}}^+ \ar@{<-->}[rrr]\ar@{^{(}->}[rd]& & & V_{L_{2N}}^+ \ar@{_{(}->}[ld] \\
\mathcal{A}_{\mathfrak{Vir},1} \ar@{<-->}@/_3.5cm/[rrrrr] \ar@{^{(}->}[r] \ar@{^{(}->}[ur] \ar@{^{(}->}[dr]&  \mathcal{A}_{\Uone}^+ \ar@{^{(}->}[u]\ar@{^{(}->}[r]\ar@{^{(}->}[d]& \mathcal{A}_{\Uone_{2N}} \ar@{<-->}[r]& V_{L_{2N}}&  M(1)^+ \ar@{^{(}->}[u]\ar@{_{(}->}[l]\ar@{^{(}->}[d] & L(1,0)\ar@{_{(}->}[l] \ar@{_{(}->}[ul] \ar@{_{(}->}[dl] \\
&\mathcal{A}_{\Uone} \ar@{<-->}[rrr]\ar@{^{(}->}[ru]& & & M(1) \ar@{_{(}->}[lu] \\
& & & & \\ 
&&&&
}
\end{equation*}

\end{section}

\noindent\textbf{Acknowledgements.} 
S.C. would like to thank Yasuyuki Kawahigashi for some useful explanations.

\end{document}